\documentclass[11pt,letterpaper]{amsart}

\usepackage{mathrsfs}
\usepackage{cite}
\usepackage{color}
\usepackage{amssymb}
\theoremstyle{plain}

\author[Van Tien Nguyen]{Van Tien Nguyen}
\address{Van Tien Nguyen, Department of Mathematics, Institute of Applied Mathematical Sciences, National Taiwan University}
\email{vtnguyen@ntu.edu.tw}

\author[Zhi-An Wang]{Zhi-An Wang}
\address{Zhi-An Wang,
Department of Applied Mathematics,
The Hong Kong Polytechnic University
Hung Hom, Kowloon, Hong Kong}
\email{mawza@polyu.edu.hk}

\author[K. Zhang]{Kaiqiang Zhang}
\address{Kaiqiang Zhang, Department of Applied Mathematics,
The Hong Kong Polytechnic University,
Hung Hom, Kowloon, Hong Kong}
\email{kaiqzhang@polyu.edu.hk}

\begin{document}

\numberwithin{equation}{section}
\renewcommand{\theequation}{\arabic{section}.\arabic{equation}}
\theoremstyle{plain}
\newtheorem{exam}{Example}[section]
\newtheorem{theorem}[exam]{Theorem}
\newtheorem{lemma}[exam]{Lemma}
\newtheorem{remark}[exam]{Remark}
\newtheorem{hyp}[exam]{Hypothesis}
\newtheorem{proposition}[exam]{Proposition}
\newtheorem{definition}[exam]{Definition}
\newtheorem{corollary}[exam]{Corollary}
\newtheorem{analytic}[exam]{Analytic extension principle}
\newtheorem{notation}[exam]{Notation}
\setlength{\baselineskip}{1.5\baselineskip}

\title[self-similar blow-up solutions for the Keller-Segel system]{{Infinitely many self-similar blow-up profiles for the Keller-Segel system in dimensions 3 to 9}}
\begin{abstract}
Based on the method of matched asymptotic expansions and Banach fixed point theorem, we rigorously construct infinitely many self-similar blow-up profiles for the  parabolic-elliptic Keller-Segel system
\begin{equation*}
\left\{\begin{array}{l}
\partial_{t} u=\Delta u-\nabla \cdot\left(u \nabla \Phi_{u}\right),  \\
0=\Delta \Phi_{u}+u,\\
u(\cdot,0)=u_0 \geq 0
\end{array}\quad \text{in}\ \mathbb{R}^{d},\right.
\end{equation*}
where $d\in \{3,\cdots,9\}$.
Our findings demonstrate that the infinitely many backward self-similar profiles approximate the rescaling radial steady-state near the origin (i.e. $0<|x|\ll1$) and $\frac{2(d-2)}{|x|^2}$ at spatial infinity (i.e. $|x|\gg1$). 
We also establish the convergence of the self-similar blow-up solutions as time tends to the blow-up time $T>0$. Our results can give a refined description of backward self-similar profiles for all $|x|\geq 0$ rather than for $0<|x|\ll1$ or $|x|\gg1$, indicating that the blow-up point is the origin and
$$
u(x,t)\sim \frac{1}{|x|^2},\ \ \ x\ne0,\  \text{as}\ t\to T.
$$
\end{abstract}

\maketitle
\section{Introduction}
This paper is concerned with the parabolic-elliptic Keller-Segel system
\begin{equation}\label{1.1}
\left\{\begin{array}{l}
\partial_{t} u=\Delta u-\nabla \cdot\left(u \nabla \Phi_{u}\right),  \\
0=\Delta \Phi_{u}+u,
\end{array}\quad \text{in}\ \mathbb{R}^{d},\right.
\end{equation}
equipped with an initial data $u(\cdot,0)=u_0$, where $d\in \{3,\cdots,9\}$.  The system \eqref{1.1} is the so-called minimal chemotaxis used to describe the chemotactic motion of mono-cellular organisms, where $u(x,t)$ represents the cell density and $\Phi_{u}$ stands for the concentration of the chemoattractant \cite{Keller-Segel}. System \eqref{1.1} also models the self-gravitating matter in stellar dynamics in astrophysical fields \cite{Wolansky}. This system has been extensively studied due to its rich biological and physical backgrounds and lot of interesting results have been obtained, e.g., see \cite{Blanchet-2008,Collot-2022, Colllot-Zhang,Davila-2020, Horstmann-2003,Horstmann-2004,Hou-Liu-Wang-Wang,Tao-Wang-2013,Mellet-2024,Winkler-2019} and references therein.


For any radial initial data $u_0\in L^\infty(\mathbb{R}^d)$, there exists a maximal time of existence $T > 0$ such that \eqref{1.1} admits a unique smooth solution on $(0,T)\times \mathbb{R}^d$, see \cite{Giga-Mizoguchi-Senba-2011}. One may refer to \cite{Bedrossian-Masmoudi-2014,Biler-2018} for other local well-posedness spaces. Due to the quadratic nature of the convective term in \eqref{1.1}, the solutions may blow up in finite time $T<+\infty$ in the sense that
$$
\limsup \limits_{t\to T}||u(t)||_{L^\infty(\mathbb{R}^d)}=+\infty.
$$
If blow-up occurs, then it holds that
$$
||u(t)||_{L^\infty(\mathbb{R}^d)}\ge (T-t)^{-1},\ \ 0<t<T,
$$
by a comparison principle. We say that the blow-up is of type I if
$$
\limsup \limits_{t\to T}(T-t)||u(t)||_{L^\infty(\mathbb{R}^d)}<\infty,
$$
otherwise, the blow-up is of type II. The blow-up set $B(u_0)$ is defined by
$$
B(u_0):=\{x_0\in\mathbb{R}^d:|u(x_j,t_j)|\to\infty\  \text{for some sequence}\ (x_j,t_j)\to(x_0,T)\},
$$
and we call $x_0$ the blow-up point.
Thanks to the divergence structure of \eqref{1.1}, the total mass of the solution is conserved in the following sense:
$$
M(u_0):=\int_{\mathbb{R}^{d}}u_0(x)dx=\int_{\mathbb{R}^{d}}u(x,t)dx,\ \ 0\le t<T.
$$

Problem \eqref{1.1} admits the  following scaling invariance: for all $a \in \mathbb{R}^{d}$ and $\lambda > 0$, the function
\begin{equation}\label{scaling}
  u_{\lambda,a}(x,t)=\frac{1}{\lambda^2}u\left(\frac{x-a}{\lambda},\frac{t}{\lambda^2}\right)
\end{equation}
also solves \eqref{1.1}. This scaling invariance gives rise to the notion of the mass-criticality in the sense that
$$ \|u_{\lambda,a}\|_{L^1(\mathbb{R}^{d})} = \lambda^{d-2}\| u\|_{L^1(\mathbb{R}^{d})},$$
by which $d = 2$ is referred to as the mass critical case, while $d = 1$ and  $d \geq 3$ the mass sub-critical and the mass super-critical cases, respectively.

The solution of \eqref{1.1} exists globally for $d=1$ as proved in \cite{Nagai-1995,Childress-Jerome}.
The critical mass threshold $8\pi$ acts as a sharp criterion separating the global existence from finite-time blow-up in the case of $d=2$, see \cite{Childress,Childress-Jerome,Diaz-1998,Blanchet-2006,Blanchet-Carrillo-Laurençot-2009}. The $8\pi$ mass threshold implies that supposing $$u_0\ge0,\ \ (1+x^2+|\ln u_0|)u_0\in L^1(\mathbb{R}^{2}),$$
 the positive solution of \eqref{1.1} blows up in finite time for $M > 8\pi$ \cite{Kurokiba-2003,Raphael-2014} and exists globally in time  for $M < 8\pi$\cite{Blanchet-2006,Dolbeault-2004}. If $M = 8\pi$, radial solutions exist globally in time \cite{Biler-2006} but infinite-time blow-up solutions with $8\pi$ mass may exist as constructed in \cite{Blanchet-2008,Davila-2020,Ghoul-2018}. For $M>8\pi$, a refined finite time blow-up profile was obtained with the form
 \begin{equation}\label{uniq}
  u(x,t)\sim\frac{1}{\lambda^2(t)}U\left(\frac{x}{\lambda(t)}\right),\ \    \lambda(t)\sim\sqrt{T-t}e^{-\sqrt{\frac{|\log(T-t)|}{2}}},
 \end{equation}
 where $U(x)=\frac{8}{(1+|x|^2)^2}$ is a steady-state solution of \eqref{1.1}, see \cite{Buseghin-2023,Raphael-2014,Herrero-Velázquez-1996,Velázquez-2002,Collot-2022}. The form \eqref{uniq} is the unique finite time blow-up behavior for radial non-negative solutions of \eqref{1.1} \cite{Mizoguchi-2022}.  An interesting phenomenon that two steady-state solutions are simultaneously collapsing and colliding is recently constructed in \cite{Collot-2025}.
 It is remarkable that any blow-up solutions are of type II for $d=2$, see  \cite{Naito-2008,Suzuki-2011}.

For $d\ge3$,  we note that the system \eqref{1.1} is referred to as the $L^1$-supercritical and $L^{d/2}$-critical since the scaling transformation \eqref{scaling} preserves the $L^{d/2}-$norm, i.e., $||u_{\lambda,a}||_{L^{\frac{d}{2}}(\mathbb{R}^{d})}=||u||_{L^{\frac{d}{2}}(\mathbb{R}^{d})}$.
Initial data with small $L^{d/2}-$norm lead to solutions that exist globally in time \cite{Corrias}. Subsequently, this result was improved in \cite{Calvez-2012} by showing that if the $L^{d/2}-$norm of initial data is less than a sharp constant derived from the Gagliardo-Nirenberg inequality, then the solution exists globally.
Large initial data give rise to finite-time blow-up \cite{Corrias,Calvez-2012,Nagai-1995}. In contrast to dimension $d=2$, the solutions of \eqref{1.1} with $d\ge3$ may blow up in finite time for an arbitrary mass since $M(u_{\lambda,a})=\lambda^{d-2}M(u)$.

Singularity formation of blow-up solutions to system \eqref{1.1} for $d\ge3$
exhibits rich dynamical behavior. When the initial data are nonnegative and radially non-increasing, it was shown in \cite{Mizoguchi-2011} that all blow-up solutions of \eqref{1.1} are of type I for $d\in[3,9]$.
A family of type I self-similar blow-up solutions was obtained by the shooting method in \cite{Michael,Herrero-1998,Naito-Senba-2012}. Remarkably, it was shown in \cite{Giga-Mizoguchi-Senba-2011} that all radial and non-negative type I blow-up solutions are asymptotically backward self-similar near the origin as $t\to T$, which signifies the significance of backward self-similar profiles for understanding the structure of singularities. A new type I-log blow-up solution of \eqref{1.1} in dimensions 3 and 4 was constructed in \cite{Nguyen-Zaag-2023}. There are also type II blow-up solutions for $d \ge3$ \cite{Herrero-Medina-1997, Collot-2023, Mizo-Senba}. The authors of \cite{Collot-2023} showed the existence and radial stability of type II blow-up solutions, characterized by mass concentrating near a sphere that shrinks to a point. This pattern, known as collapsing-ring blow-up, also emerges in the nonlinear Schr$\mathrm{\ddot{o}}$dinger equation \cite{Merle-2014,Fibich-2007}. For $d\ge 11$, type II solutions concentrating at a steady-state solution are constructed in \cite{Mizo-Senba}. This paper is concerned with type I blow-up solutions.

 Backward self-similar solutions of \eqref{1.1} are of the form
\begin{equation}\label{gmfs}
u(x,t)=\frac{1}{T-t}U\left(y\right),\ \ y=\frac{x}{\sqrt{T-t}},
\end{equation}
where $U(y)$ is the backward self-similar profile satisfying
\begin{equation}\label{self-similar}
\Delta U-\frac{y\cdot\nabla U}{2}-U-\nabla\cdot(U\nabla\Phi_{U})=0,\ \Delta \Phi_U+U=0.
\end{equation}
We denote $r=|y|$. In the radial case, for $d\ge1$, there holds
$$
\partial_r\Phi_U(r)=-\frac{1}{r^{d-1}}\int_0^rU(s)s^{d-1}ds.
$$
Then the equation \eqref{self-similar} can be written in the radial form
\begin{equation}\label{zq11}
\partial_{rr} U+\frac{d-1}{r}\partial_rU-\frac{1}{2}r\partial_rU-U+U^2+\left(\frac{1}{r^{d-1}}\int_0^rU(s)s^{d-1}ds\right)\partial_r U=0.
\end{equation}
 There are four known classes of solutions of \eqref{zq11}:
\begin{itemize}
    \item For $d\ge1$, the constant solutions
    \begin{equation}\label{explicit1}
     \bar{U}_0=0,\ \ \bar{U}_1=1.
    \end{equation}
     \item For $d\ge3$, the solution singular at the origin
      \begin{equation}\label{explicit2}
     \bar{U}_2=\frac{2(d-2)}{r^2}.
     \end{equation}
     \item For $d\ge3$, the explicit smooth positive solution \cite{Michael}
     \begin{equation}\label{explicit4}
     \bar{U}_3= \frac{4(d-2)(2d+r^2)}{(2(d-2)+r^2)^2}.
     \end{equation}
     \item For $d\in[3,9]$, there exists a countable family of positive smooth radially symmetric solutions $\{\bar{U}_n\}_{n\ge4}$ \cite{Herrero-1998,Michael,Naito-Senba-2012}, where
     \begin{equation}\label{explicit3}
     \bar{U}_n\sim\frac{1}{r^2},\ \ \text{as}\ \ r\to +\infty.
     \end{equation}
\end{itemize}

 With the shooting method,  a family of radially symmetric solutions $\{\bar{U}_n\}_{n\ge4}$ has been constructed in \cite{Herrero-1998} for $d=3$ and in \cite{Michael,Naito-Senba-2012}  for $3\le d\le 9$. For $d=3$, it was shown in
 \cite{Glogic} that $\bar{U}_3$ is a stable self-similar profile based on the semigroup approach. Very recently, the non-radial stability of $\bar{U}_3$ was proved in \cite{Li-Zhou-2025}. For $d\ge3$, it was proved in \cite{Colllot-Zhang} that all the fundamental self-similar profiles $\{\bar{U}_n\}_{n\ge3}$ are conditionally stable (of finite co-dimension).

Backward self-similar profiles of \eqref{1.1} (i.e. the solutions of \eqref{self-similar}) are still not completely classified, even in the radial setting. Accurately describing the self-similar profiles is a crucial step in classifying all possible blow-up profiles for \eqref{1.1} (at least in the radial case). 

This paper aims to construct more precise backward self-similar profiles by using different approaches. We recall some results below in connection with our work. For $d=3$, the
 authors of \cite{Herrero-1998} showed that there exists a sequence of self-similar profiles (i.e. solutions of \eqref{zq11}), denoted by  $\{G_n(r)\}_{n\geq 1}$, which satisfy
 $$
 G_n(r)\sim K_n\ \text{as}\ r\to0,\ \ \lim_{r\to\infty}G_n(r)=\frac{A_n}{r^2},
 $$
 where $K_n>0$, $A_n$ are constants, and $\lim \limits_{n\to +\infty}K_n=\infty$.
Subsequently, for $3\le d\le9$,
it was shown in \cite{Michael} that there exists a countable number of self-similar profiles $\{\bar{G}_n\}_{n\geq 1}$ satisfying
 $$
 \bar{G}_n(r)\lesssim 1\ \text{as}\ r\to0,\ \lim_{r\to\infty}\bar{G}_n(r)=\frac{c_n}{r^2},\ \text{for  some constant}\ c_n\in(0,2].
 $$
The works \cite{Michael, Herrero-1998} discovered two essential common properties for the family of self-similar profiles for fixed $n$, that is they are bounded as $0<r\ll1$  and behave like $\frac{1}{r^2}$ as $r\gg 1$.
In another work \cite{Naito-Senba-2012}, for $3\le d\le9$, the authors proved that there exist a countable number of self-similar profiles $\{\tilde{G}_n(r)\}_{n\geq 1}$ which are bounded near the origin for every $n\geq 1$ and
\begin{equation}\label{asy1}
\lim_{n\to\infty}\tilde{G}_n(0)=+\infty,\ \lim_{n\to\infty}\tilde{G}_n(r)=\frac{2(d-2)}{r^2}\ \text{for}\ r>0.
\end{equation}
The work \cite{Naito-Senba-2012} gave an asymptotic description of self-similar profiles as $n \to \infty$.  For fixed $n\geq 1$, the self-similar profiles were precisely described only for $r\gg1$  in \cite{Michael, Herrero-1998}, while the precise description of self-similar profiles for $r>0$ not large are unavailable. Recently, for $d\geq 3$, self-similar profiles of blow-up solutions to \eqref{1.1} were shown to behave like $\frac{1}{r^{2}}$ for $0<r\ll 1$ for a certain class of radially non-increasing initial data in \cite{Zhou-2024} by the zero number argument, answering an open question in \cite{Souplet-Winkler}. In this paper, by using a different approach, namely the method of matched asymptotic expansions and the Banach fixed point theorem, we obtain a precise description of self-similar profiles $U_n(r)$ for all $r \in [0,\infty)$, as described in \eqref{eq:Un} below.

To state our result, we first present the asymptotic behavior of  steady-state solution of \eqref{1.1}.
Let $Q(r)$ be the unique solution to
\begin{equation}\label{steady-equation11}
\left\{\begin{aligned}
&\partial_{rr} Q+\frac{d-1}{r}\partial_rQ+Q^2+\partial_rQ\frac{1}{r^{d-1}}\int_0^rQ(s)s^{d-1}ds=0, \\
&Q(0)=1,\ \ Q'(0)=0.
\end{aligned} \right.
\end{equation}
It is clear that $Q(r)$ is a radial steady-state solution of \eqref{1.1} with $r=|x|$. It will be shown in Section \ref{Construction} that the asymptotic behavior of $Q$ is
$$
Q(r)=\frac{2(d-2)}{r^2}+O(r^{-\frac{5}{2}}),\ \ \text{as}\ \ r\to+\infty,
$$
where $Q=2d\bar{Q}+2r\partial_r\bar{Q}$ and the asymptotic profile of $\bar{Q}$ as $r \to \infty$ is given in \eqref{steady}.
Our main results are stated as follows.
\begin{theorem}\label{Main-theorem}
For $3\le d\le 9$, there exist infinitely many smooth radially symmetric solutions $U_n(y)$ ($n\in \mathbb{N}$) to the self-similar equation \eqref{self-similar}.
 Moreover, there exists a sufficiently small constant $r_0>0$ independent of $n$ such that the following results hold.
 \begin{itemize}
 \item[1.] (Profiles near the origin). There exists a sequence $\mu_n>0$ with
    $
    \lim \limits_{n\to +\infty}\mu_n=0
    $
    such that
   \begin{equation}\label{wat2}
       \lim_{n\to+\infty}\sup_{r\le r_0}\left|U_n(r)-\frac{1}{\mu_n^2}Q\left(\frac{r}{\mu_n}\right)\right|=0.
   \end{equation}
   \item[2.] (Profiles away from the origin). As $r\geq r_0$, $U_n(r)$ satisfies
\begin{equation}\label{wat3}
\lim_{n\to+\infty}\sup_{r\ge r_0}(1+r^2)\left|U_n(r)-\frac{2(d-2)}{r^2}\right|=0.
\end{equation}
\end{itemize}
For any $0<T<+\infty$, the solution of \eqref{1.1} with initial data $u_0=\frac{1}{T}U_n(\frac{x}{\sqrt{T}})$ blows up at  time $T$ with
$$
u(x,t)=\frac{1}{T-t}U_n\left(\frac{x}{\sqrt{T-t}}\right),
$$
where the blow-up is of type I and $B(u_0)=0$. Moreover, there exists a function $u^*(x) \sim \frac{1}{|x|^2}$ such that
$\lim\limits_{t \to T} u(x,t)=u^*(x)$ for all $|x|>0$ and
 \begin{equation}\label{dzj}
 \lim_{t\to T}||u(\cdot,t)-u^*(\cdot)||_{L^p(\mathbb{R}^d)}=0, \ \forall \ p\in[1,\frac{d}{2}).
\end{equation}
\end{theorem}

\begin{remark}\label{rem}
Based on the proof of Theorem \ref{Main-theorem},  the profile of the solutions $U_n$ of \eqref{self-similar}, as constructed in Theorem \ref{Main-theorem}, can be more precisely described as follows. First, we define
$$\mathcal{U}=2d\tilde{u}_1+2r\partial_r\tilde{u}_1$$
where $\tilde{u}_1:=u_1$ is a known function  for $d=3$ (see Lemma  \ref{ODE}\footnote{The definitions of $\tilde{u}_1$ for $d\in[4,9]$ are obtained by the same process as in Lemma \ref{ODE}.}).
Then there exist $$0<r_0\ll1,\ 0<\mu_n<r_0,\  0<\varepsilon(\mu_n)\ll r_0^{\frac{1}{2}}$$ with $\lim \limits_{n\to +\infty}\mu_n=0,\  \lim \limits_{n\to +\infty}\varepsilon(\mu_n)=0,$
and
$$
\tilde{\mathcal{U}}\in \tilde{X}_{r_0},\ \ \tilde{Q}\in \tilde{Y}_{\frac{r_0}{\mu_n}}
$$
where the definitions
of the spaces $\tilde{X}_{r_0}$, $\tilde{Y}_{r}$ are given in  \eqref{swiland} and \eqref{jsj} for $d=3$, respectively\footnote{The definitions of the spaces $\tilde{X}_{r_0}$, $\tilde{Y}_{r}$ for $d\in[4,9]$ are similar by the same process of the proof for $d=3$.},
such that
\begin{equation}\label{eq:Un}
U_{n}(r):=\begin{cases}
\big(\frac{Q}{\mu_n^2}+\mu_n^2\tilde{Q})\big(\frac{r}{\mu_n}\big) &\text{for }0\leq r\leq r_0,\vspace{2ex}\\ \frac{2(d-2)}{r^2}+\varepsilon(\mu_n)(\mathcal{U}+\tilde{\mathcal{U}})(r)  &\text{for }r>r_0,\end{cases}
\end{equation}
solves \eqref{zq11}.
\end{remark}

By \eqref{eq:Un} we obtain a precise description of self-similar profiles $U_n(r)$ for all $r \in [0,\infty)$. In particular, we show that $U_n(r)$ behaves like the rescaled steady-state solutions $\frac{1}{\mu^2_n}Q(\frac{r}{\mu_n})$ for $0\le r\ll1$ and
 $U_n(r)\sim \frac{2(d-2)}{r^2}$ for $r\gg1$.
 For $3<d\leq 9$, we know from \eqref{eq:Un} that the profiles obtained in this paper are different from those in \cite{Michael} since $2(d-2)>2$, but have the same asymptotic properties as in \eqref{asy1}  as $n \to \infty$. Whether the self-similar profiles constructed in \cite{Herrero-1998,Naito-Senba-2012} and in Theorem \ref{Main-theorem} are equivalent is an interesting open question.

For $d=2$, the limiting spatial profile of radial blow-up solutions to \eqref{1.1} resembles a Dirac mass perturbed by a $L^1$ function, i.e.,
\begin{equation}\label{2d}
u(\cdot,t)\rightharpoonup8\pi\delta_0+f\ \ \text{in}\ C_0(\mathbb{R}^2)^*\ \ \text{as}\ t\to T,
\end{equation}
where $0\le f\in L^1(\mathbb{R}^2)$, see \cite{Herrero-Velázquez-1996,Herrero-Velázquez-1997}. In contrast, for $d\in [3,9]$, as seen from  \eqref{eq:Un},  our result shows that there exist radial solutions of \eqref{1.1} that satisfy
$$
u(x,t)\sim 1/|x|^2,\ \ \ x\ne0,\  \text{as}\ t\to T.
$$
which is quite different from the case $d=2$ in \eqref{2d}.

\begin{remark}[Finite codimensional radial stability]
The stability of self-similar blow-up profiles constructed in \cite{Herrero-1998,Michael}  was established in \cite{{Glogic}, Colllot-Zhang}.  Using the same ideas of  \cite{Colllot-Zhang}, one can also show that
the profiles constructed in Theorem \ref{Main-theorem} are stable along a set of radial initial data with finite Lipschitz codimension equal to the number of unstable eigenmodes.
The non-radial stability of self-similar profiles is still an open problem as far as we know.
\end{remark}
\textbf{Organization of the paper.}
In Section \ref{Construction}, we first introduce a key transformation which converts \eqref{self-similar} into a local elliptic equation in $\mathbb{R}^{d+2}$. Then using the method of matched asymptotic expansions, we rigorously derive a sequence of smooth self-similar profiles. In Section \ref{Proofofmainresults}, we give a complete proof for Theorem \ref{Main-theorem}.

\section{Construction of self-similar profiles}\label{Construction}
We start by introducing some notations.

\noindent \textbf{Notation.} We write $a\lesssim b$, if there exists $c>0$ such that $a\le c b$, and $a\sim b$ if simultaneously $a\lesssim b$ and $b\lesssim a$. If the inequality $|f| \le C|g|$ holds for some constant $C>0$, then we write $f=O(g)$.
\subsection{Key results}\label{Preliminaries}
Our main goal is to derive the radial self-similar profile $U(r):=U(|y|)$ which satisfies \eqref{zq11}. To study the nonlocal equation \eqref{zq11}, we introduce the following so-called reduced mass (cf. \cite{Michael}),
\begin{equation}\label{wphi}
   \Phi(r)=\frac{1}{2 r^{d}} \int_{0}^{r} U(s) s^{d-1} ds,
\end{equation}
and transform \eqref{zq11} into a local equation for $\Phi(r)$ satisfying
$$
\partial_{rr}\Phi+\frac{d+1}{r}\partial_{r}\Phi-\Phi-\frac{r\partial_{r}\Phi}{2}+2d\Phi^2+2r\Phi\Phi_r=0.
$$
Clearly, $\Phi(r)$ is the radially symmetric solution of
\begin{equation}\label{hl}
 \Delta \Phi-\frac{1}{2}\Lambda\Phi+2d\Phi^2+y\cdot\nabla(\Phi^2)=0,\ \ \ y\in\mathbb{R}^{d+2},
\end{equation}
with $\Lambda $ being a differential operator defined by
$$
\Lambda u:=2u+y\cdot\nabla u.
$$
By \eqref{explicit1}, for $d\ge1$,
\eqref{hl} admits constant solutions
$
\bar{\Phi}_0=0,\ \bar{\Phi}_1=\frac{1}{2d}.$
By \eqref{explicit2} and \eqref{explicit4}, for $d\ge3$, \eqref{hl} admits explicit radial solutions
\begin{equation}\label{class1}
\bar{\Phi}_2=\frac{1}{|y|^2},\ \bar{\Phi}_3=\frac{2}{2(d-2)+|y|^2}.
\end{equation}
From \eqref{explicit3}, for $d\in[3,9]$,  there exists a countable family of positive smooth radially symmetric solutions $\{\bar{\Phi}_n\}_{n\ge4}$ of \eqref{hl} such that
\begin{equation}\label{class2}
\bar{\Phi}_n\sim\frac{1}{|y|^2}\  \text{as}\ |y|\to +\infty.
\end{equation}
The main result of this paper, as stated in Theorem \ref{Main-theorem} along with Remark \ref{rem}, consists of the construction of a class of more general solutions than those given in \eqref{class1} and \eqref{class2}, but share some similar properties when $0<|y|\ll1$ or $|y|\gg 1$.

The rest of this paper is focused on the case $d=3$ for the simplicity of presentation.
The extension of the result to $d\in[4,9]$\footnote{This oscillating behavior exists when the differential equation $x^2+(d+2)x+4(d-1)=0$ has complex roots, which holds in the case $d\in[3,9].$ } is straightforward since the oscillating behavior of the radial steady-state profile $Q=2d\bar{Q}+2r\partial_r\bar{Q}$ for $d=3$ (see \eqref{steady} for the definition of $\bar{Q}$) also exists for $d\in[4,9]$.  As in \cite{Budd-Norbury, Dancer-Guo-Wei, Collot-Pierre-2019}, the matching of exterior solutions with interior solutions can be obtained by this oscillating behavior.

When $d=3$, equation \eqref{hl} is reduced to
\begin{equation}\label{swin}
    \Delta \Phi-\frac{1}{2}\Lambda \Phi+6\Phi^2+y\cdot\nabla(\Phi^2)=0,\ \ \ \ y\in \mathbb{R}^5.
\end{equation}
Applying the transformation \eqref{wphi}, we then obtain the radially symmetric solution of \eqref{self-similar} as follows
$$U=6 \Phi+2r \partial_r \Phi.$$
We define $$\Phi_*:=\bar{\Phi}_2=\frac{1}{r^2},\ \  \bar{Q}(r)=\frac{1}{2r^3}\int_0^rQ(s)s^2ds,$$
where $Q$ is given by \eqref{steady-equation11}.

The following is the key proposition of this paper, from which  Theorem \ref{Main-theorem} directly follows.
\begin{proposition}\label{swin1}
There exist infinitely many smooth radially symmetric solutions $\Phi_n$ ($n\in\mathbb{N}$) to equation \eqref{swin}.
    Moreover, there exists a sufficiently small constant $r_0>0$ which is independent of $n$ such that the following results hold.\\
    1. (Behavior near the origin). There exists a sequence $\mu_n>0$ with
    $
    \lim \limits_{n\to +\infty}\mu_n=0
    $
    such that
    \begin{equation}\label{is}
        \lim_{n\to+\infty}\sup_{r\le r_0}\left|\Phi_n-\frac{1}{\mu_n^2}\bar{Q}\left(\frac{r}{\mu_n}\right)\right|=0.
    \end{equation}
   2. (Behavior away from the origin).  As $r\geq r_0$, $\Phi_n(r)$ satisfies
   \begin{equation}\label{is1}
   \lim_{n\to+\infty}\sup_{r\ge r_0}(1+r^2)|\Phi_n-\Phi_*|=0.
   \end{equation}
\end{proposition}
The remainder of this section is devoted to proving the above proposition.

\subsection{Exterior profiles}\label{section3}
The aim of this subsection is to construct a radial solution to \eqref{swin} on $[r_0,+\infty)$, where $0<r_0<1$.
We are initially concerned with the asymptotic behavior of the fundamental solutions for the equation
$
L(u)=0
$ on ($0,+\infty$),
where $L$ is the linearized operator of \eqref{swin} around $\Phi_*$, defined as
\begin{equation}\label{L}
L=-\Delta+\frac{1}{2}\Lambda-2y\cdot\nabla(\Phi_*\cdot)-12\Phi_*.
\end{equation}
Given $0<r_0<1$, we define $X_{r_0}$ as the space of continuous functions on $[r_0,+\infty)$ such
that the following norm is finite
\begin{equation}\label{swiland}
||w||_{X_{r_0}}=\sup_{r_0\le r\le1}(r^{\frac{5}{2}}|w|+r^{\frac{7}{2}}|\partial_rw|)
+\sup_{r\ge 1}\left(r^{4}|w|+r^{5}|\partial_rw|\right).
\end{equation}

\begin{lemma}\label{ODE}
Let $L$ be defined in \eqref{L}. Then the following results hold.
\\
1. The basis of the fundamental solutions:
The equation
$$
L(u)=0\ \ \text{on}\ \ (0,+\infty)
$$
has two fundamental solutions $u_i\ (i=1,2)$ with the following asymptotic behavior as $r\to \infty$:
\begin{equation}\label{sap}
u_1(r)= r^{-2}(1+O(r^{-2})) \quad \mbox{and} \quad u_2(r)=r^{-5}e^{\frac{r^2}{4}}(1+O(r^{-2})),
\end{equation}
and as $r\to 0$:
\begin{equation}\label{origin}
u_1(r)=\frac{c_1\sin(\frac{\sqrt{7}}{2}\log(r)+c_2)}{r^{\frac{5}{2}}}+O(r^{-\frac{1}{2}}) \quad \mbox{and} \quad u_2(r)=\frac{c_3\sin(\frac{\sqrt{7}}{2}\log(r)+c_4)}{r^{\frac{5}{2}}}+O(r^{-\frac{1}{2}}),
\end{equation}
where $c_1$, $c_3\ne0$ and $c_2$, $c_4\in\mathbb{R}$.\\
2. The continuity of the resolvent: The inverse
\begin{equation}\label{zj}
\tau{(f)}=\left(\int_r^{+\infty}fu_2{s}^6e^{-\frac{{s}^2}{4}}ds\right)u_1-\left(\int_r^{+\infty}fu_1{s}^6e^{-\frac{{s}^2}{4}}ds\right)u_2
\end{equation}
satisfies
$
L(\tau{(f)})=f
$
and
\begin{equation}\label{motors}
||\tau{(f)}||_{X_{r_0}}\lesssim \int_{r_0}^{1}|f|s^{\frac{7}{2}}ds+\sup_{r\ge 1}r^{4}|f|.
\end{equation}
\end{lemma}
\begin{proof}
\emph{$\mathbf{Step\ 1}.$}\  Basis of homogeneous solutions.
 We define the changing of variable
\begin{equation}\label{ysnn}
u(r)=\frac{1}{z^{\frac{\gamma}{2}}}\phi(z),\qquad z=r^2,
\end{equation}
where $\gamma$ satisfies
$
-\gamma^2+5\gamma-8=0.
$
From
$$
\partial_r=2r\partial_z,\ \partial_{rr}=4z\partial_{zz}+2\partial_z,\ r\partial_r=2z\partial_z,
$$
one has
\begin{equation}\nonumber
\begin{aligned}
L(u)&=\left(-4z\partial_{zz}-2\partial_z-8\partial_z+z\partial_z+1-8\Phi_*-4\Phi_*z\partial_z\right)\left(\frac{1}{z^{\frac{\gamma}{2}}}\phi(z)\right)\\
&=\frac{1}{z^{\frac{\gamma}{2}}}\bigg\{-4z\phi''(z)+
(4\gamma-14+z)\phi'(z)+\left[1-\frac{\gamma}{2}+\frac{1}{z}(-\gamma^2+5\gamma-8)\right]\phi\bigg\}\\
&=\frac{1}{z^{\frac{\gamma}{2}}}\bigg\{-4z\phi''(z)+
(4\gamma-14+z)\phi'(z)+(1-\frac{\gamma}{2})\phi\bigg\}.
\end{aligned}
\end{equation}
$
$
Let
$
\phi(z)=\nu(\xi)$ and $ \xi=\frac{z}{4}.
$
Then,
$$
L(u)=-\frac{1}{z^{\frac{\gamma}{2}}}\left\{\xi\nu''(\xi)+\left(-\gamma+\frac{7}{2}-\xi\right)\nu'(\xi)+(\frac{\gamma}{2}-1)\nu(\xi)\right\}
.
$$
Therefore, $L(u)=0$ if and only if
\begin{equation}\label{fervex}
\xi\frac{d^2\nu}{d\xi^2}+(b-\xi)\frac{d\nu}{d\xi}-a\nu=0,
\end{equation}
where $$b=\frac{7}{2}-\gamma,\ a=1-\frac{\gamma}{2}.$$ The equation \eqref{fervex} is known as the well studied Kummer's equation (see \cite{Olver-Lozier}). If the parameter $a$ is not a negative integer (which holds in particular for our case), then the fundamental solutions to Kummer's equation consists of the Kummer function $M(a,b,\xi)$ and the Tricomi function $U(a,b,\xi)$. Therefore, $\nu(\xi)$ is a linear combination of the special functions  $M(a,b,\xi)$ and $U(a,b,\xi)$, whose asymptotic profiles at infinity are given by
\begin{equation}\label{large}
    M(a,b,\xi)=\frac{\Gamma(b)}{\Gamma(a)}\xi^{a-b}e^\xi(1+O(\xi^{-1})),\ \ U(a,b,\xi)=\xi^{-a}(1+O(\xi^{-1}))\ \ \text{as}\ \xi\to+\infty.
\end{equation}
Then by \eqref{ysnn} and \eqref{large}, one obtains \eqref{sap}.

For the behavior near the origin, we have
\begin{equation}\label{complex}
    M(a,b,\xi)=1+O(\xi)\ \ \text{as}\ \xi\to0.
\end{equation}
It is easy to check that the real part of $b$ satisfies $\mathcal{R}(b)=1$ ($b\ne1$). Then it follows that
\begin{equation}\label{complex1}
    U(a,b,\xi)=\frac{\Gamma(b-1)}{\Gamma(a)}\xi^{1-b}+\frac{\Gamma(1-b)}{\Gamma(a-b+1)}+O(\xi)\ \ \text{as}\ \xi\to0.
\end{equation}
Since the polynomial
$
\gamma^2-5\gamma+8=0
$
has complex roots
$
   \gamma=\frac{5}{2}\pm\frac{\sqrt{7}i}{2},
$
then combining \eqref{ysnn}, \eqref{complex} and \eqref{complex1}, one obtains \eqref{origin}.
\\
\noindent \emph{$\mathbf{Step\ 2}.$}\  Estimate on the resolvent. The Wronskian
$
W:=u_1'u_2-u_2'u_1
$
satisfies
$
W'=\left(\frac{r}{2}-\frac{6}{r}\right)W$, and $W=\frac{C}{r^6}e^{\frac{r^2}{4}}.
$
We may assume $C=1$ without loss of generality. Next, we solve
$
L(w)=f.$
By the variation of constants, we obtain
$$
w=\left(a_1+\int_r^{+\infty}fu_2{s}^6e^{-\frac{{s}^2}{4}}ds\right)u_1+\left(a_2-\int_r^{+\infty}fu_1{s}^6e^{-\frac{{s}^2}{4}}ds\right)u_2,\ \ a_1,\ a_2\in\mathbb{R}.
$$
Then, $\tau(f)$ satisfies
$
L(\tau(f))=f
$
by choosing $a_1=a_2=0$ in the above.

Next, we estimate the asymptotic behavior of $\tau(f)$. For $r\ge 1$, we have
\begin{equation}\label{asym}
\begin{aligned}
r^{4}|\tau(f)|
&=r^{4}\left|\left(\int_r^{+\infty}fu_2{s}^6e^{-\frac{{s}^2}{4}}ds\right)u_1-\left(\int_r^{+\infty}fu_1{s}^6e^{-\frac{{s}^2}{4}}ds\right)u_2\right|\\
&\lesssim r^2\left(\int_r^{+\infty}|f|{s}ds\right)+r^{-1}e^{\frac{r^2}{4}}\left(\int_r^{+\infty}|f|{s}^{4}e^{-\frac{{s}^2}{4}}ds\right)\\
&\lesssim \sup_{r\ge 1}r^{4}|f|\left\{\left(\int_r^{+\infty}\frac{ds}{s^3}\right)r^2+
r^{-1}e^{\frac{r^2}{4}}\left(\int_r^{+\infty}e^{-\frac{{s}^2}{4}}ds\right)\right\}\\
&\lesssim \sup_{r\ge 1}r^{4}|f|,
\end{aligned}
\end{equation}
and
\begin{equation}\label{asym1}
\begin{aligned}
r^{5}|\partial_r\tau(f)|
&=r^{5}\left|\left(\int_r^{+\infty}fu_2{s}^6e^{-\frac{{s}^2}{4}}ds\right)\partial_ru_1-\left(\int_r^{+\infty}fu_1{s}^6e^{-\frac{{s}^2}{4}}ds\right)\partial_ru_2\right|\\
&\lesssim r^2\left(\int_r^{+\infty}|f|{s}ds\right)+(r^{-1}+r)e^{\frac{r^2}{4}}\left(\int_r^{+\infty}|f|{s}^{4}e^{-\frac{{s}^2}{4}}ds\right)\\
&\lesssim \sup_{r\ge 1}r^{4}|f|\left\{\left(\int_r^{+\infty}\frac{ds}{s^3}\right)r^2+
(r^{-1}+r)e^{\frac{r^2}{4}}\left(\int_r^{+\infty}e^{-\frac{{s}^2}{4}}ds\right)\right\}\\\
&\lesssim \sup_{r\ge 1}r^{4}|f|.
\end{aligned}
\end{equation}
For $r_0\le r\le1$, by \eqref{origin} and \eqref{asym}, we have
\begin{equation}\label{asym2}
\begin{aligned}
r^{\frac{5}{2}}|\tau(f)|&\le r^{\frac{5}{2}}\left|\left(\int_{r}^{1}fu_2{s}^6e^{-\frac{{s}^2}{4}}ds\right)u_1-\left(\int_{r}^{1}fu_1{s}^6e^{-\frac{{s}^2}{4}}ds\right)u_2\right|\\
&\ \ +r^{\frac{5}{2}}\left|\left(\int_1^{+\infty}fu_2{s}^6e^{-\frac{{s}^2}{4}}ds\right)u_1-\left(\int_1^{+\infty}fu_1{s}^6e^{-\frac{{s}^2}{4}}ds\right)u_2\right|\\
&\lesssim \int_{r_0}^{1}|f|s^{\frac{7}{2}}ds+\sup_{r\ge 1}r^{4}|f|.
\end{aligned}
\end{equation}
Similarly, for $r_0\le r\le1$, by \eqref{origin}, \eqref{asym1} and \eqref{asym2}, we have
\begin{equation}\label{asym3}
\begin{aligned}
r^{\frac{7}{2}}|\partial_r\tau(f)|
&=r^{\frac{7}{2}}\left|\left(\int_r^{+\infty}fu_2{s}^6e^{-\frac{{s}^2}{4}}ds\right)\partial_ru_1-\left(\int_r^{+\infty}fu_1{s}^6e^{-\frac{{s}^2}{4}}ds\right)\partial_ru_2\right|\\
&\lesssim r^{\frac{7}{2}}\left|\left(\int_r^{1}fu_2{s}^6e^{-\frac{{s}^2}{4}}ds\right)\partial_ru_1-\left(\int_r^{1}fu_1{s}^6e^{-\frac{{s}^2}{4}}ds\right)\partial_ru_2\right|\\
&\ \ +r^{\frac{5}{2}}\left|\left(\int_1^{+\infty}fu_2{s}^6e^{-\frac{{s}^2}{4}}ds\right)\partial_ru_1-\left(\int_1^{+\infty}fu_1{s}^6e^{-\frac{{s}^2}{4}}ds\right)\partial_ru_2\right|\\
&\lesssim\int_{r_0}^{1}|f|s^{\frac{7}{2}}ds+\sup_{r\ge 1}r^{4}|f|.
\end{aligned}
\end{equation}
Then \eqref{motors} is obtained by combining \eqref{asym}, \eqref{asym1},
\eqref{asym2} and \eqref{asym3}.
\end{proof}

We construct a outer solutions of the self-similar equation in the following.
\begin{proposition}\label{taxi}
Let $0<r_0\ll1$.
For any
$
0<\varepsilon\ll r_0^{\frac{1}{2}},
$
there exists a radial solution to
\begin{equation}\label{self-similar-equation}
    \Delta \Phi-\frac{1}{2}\Lambda \Phi+6\Phi^2+y\cdot \nabla( \Phi^2)=0, \ \ \text{on}\ \ [r_0,+\infty)
\end{equation}
with the form
$$
\Phi=\Phi_*+\varepsilon u_1+\varepsilon w,
$$
with
\begin{equation}\label{swi}
    ||w||_{X_{r_0}}\lesssim\varepsilon r_0^{-\frac{1}{2}},\ \ w|_{\varepsilon=0}=0,
\
||\partial_\varepsilon w||_{X_{r_0}}\lesssim r_0^{-\frac{1}{2}}.
\end{equation}
\end{proposition}
\begin{proof}
\textbf{Step 1.} Fixed point argument. Let
$
\Phi=\Phi_*+\varepsilon v
$
satisfy \eqref{self-similar-equation} for $r\ge r_0$. Then
$$
L(v)=\varepsilon(y\cdot\nabla(v^2)+6v^2).
$$
We set
$
v=u_1+w.
$
Since $L(u_1)=0$, then $w$ satisfies
$$
L(w)=\varepsilon (y\cdot\nabla(u_1+w)^2+6(u_1+w)^2),\ \ \forall\ r\ge r_0.
$$
Next, we find the solution of
\begin{equation}\label{eq}
   w=\varepsilon\tau(G[u_1]w),
\end{equation}
where $\tau(f)$ is defined in \eqref{zj} and
$$
G[u_1]w=r\partial_r(u_1+w)^2+6(u_1+w)^2.
$$
We claim the following estimates: if
$
||w_i||_{X_{r_0}}\le1,\ i=1,2,
$
then
\begin{equation}\label{fixed-point}
   \int_{r_0}^1 |G[u_1]w_i|s^{\frac{7}{2}}ds+\sup_{r\ge1}r^4|G[u_1]w_i|\lesssim r_0^{-\frac{1}{2}},\ \ i=1,2,
\end{equation}
and
\begin{equation}\label{fixed-point1}
\int_{r_0}^1 |G[u_1]w_1-G[u_1]w_2|s^{\frac{7}{2}}ds+\sup_{r\ge1}r^4|G[u_1]w_1-G[u_1]w_2|\lesssim r_0^{-\frac{1}{2}}||w_1-w_2||_{X_{r_0}}.
\end{equation}
If  $\varepsilon r_0^{-\frac{1}{2}}\ll1$,  and \eqref{fixed-point}-\eqref{fixed-point1} hold, by the continuity estimate on the resolvent \eqref{motors} and the Banach fixed theorem, there exists a unique solution to \eqref{eq} with
$
||w||_{X_{r_0}}\lesssim \varepsilon r_0^{-\frac{1}{2}}.
$
We know from \eqref{eq} that $w|_{\varepsilon=0}=0$ and
$
\partial_\varepsilon w=\tau(G[u_1]w).
$
Then by \eqref{motors} and \eqref{fixed-point}, we get
$$
||\partial_\varepsilon w|||_{X_{r_0}}=||\tau(G[u_1]w)||_{X_{r_0}}\lesssim r_0^{-\frac{1}{2}}.
$$
{\textbf{Step 2.}} Proof of estimates \eqref{fixed-point} and \eqref{fixed-point1}.
By \eqref{origin} and the definition of $X_{r_0}$ in \eqref{swiland}, for $w\in X_{r_0}$ and $r_0\le r\le1$, we have
\begin{equation}\label{pf}
|w(r)|+|u_1(r)|+|r\partial_r(w+u_1)|\lesssim r^{-\frac{5}{2}},
\end{equation}
while for $r\ge1$,
\begin{equation}\label{pf1}
|w(r)|+|u_1(r)|+|r\partial_r(w+u_1)|\lesssim r^{-2}.
\end{equation}
Next, we prove \eqref{fixed-point}. For $r_0\le r\le1$, by \eqref{pf}, we have
\begin{equation}\label{a}
\begin{aligned}
\int_{r_0}^1 |G[u_1]w|s^{\frac{7}{2}}ds&=\int_{r_0}^1\bigg(|s\partial_{s}(u_1+w)^2|
+6(u_1+w)^2\bigg)s^{\frac{7}{2}}ds\lesssim\int_{r_0}^1s^{-\frac{3}{2}}ds\lesssim r_0^{-\frac{1}{2}}.
\end{aligned}
\end{equation}
For $r\ge1$, by \eqref{pf1}, we have
$
|G[u_1]w|=r\partial_r(u_1+w)^2+6(u_1+w)^2\lesssim r^{-4},
$
and hence
\begin{equation}\label{a1}
\sup_{r\ge1}r^4|G[u_1]w|\lesssim1.
\end{equation}
We conclude the proof of \eqref{fixed-point} by \eqref{a} and \eqref{a1}.

Next, we prove \eqref{fixed-point1}. For
$
w_i\in X_{r_0}\ (i=1,2),
$
we have
\begin{equation}\nonumber
\begin{aligned}
G[u_1]w_1-G[u_1]w_2=r\partial_r[(2u_1+w_1+w_2)(w_1-w_2)]+6(2u_1+w_1+w_2)(w_1-w_2).
\end{aligned}
\end{equation}
For $r\ge1$, by \eqref{sap} and the definition of $X_{r_0}$ in \eqref{swiland}, we get
\begin{equation}\label{x}
    |6(2u_1+w_1+w_2)(w_1-w_2)|\lesssim |w_1-w_2|,
\end{equation}
and
\begin{equation}\label{x2}
(r\partial_r+1)(2u_1+w_1+w_2)\lesssim1.
\end{equation}
By \eqref{x2}, we obtain
\begin{equation}\nonumber
\begin{aligned}
&r\partial_r[(2u_1+w_1+w_2)(w_1-w_2)]\\
&=[r\partial_r(2u_1+w_1+w_2)](w_1-w_2)+[r\partial_r(w_1-w_2)](2u_1+w_1+w_2)\\
&\lesssim |w_1-w_2|+ r\partial_r|w_1-w_2|.
\end{aligned}
\end{equation}
Then combining \eqref{x}, we have
\begin{equation}\nonumber
\begin{aligned}
|G[u_1]w_1-G[u_1]w_2|&\le  |6(2u_1+w_1+w_2)(w_1-w_2)|+|r\partial_r[(2u_1+w_1+w_2)(w_1-w_2)]|\\
&\lesssim r\partial_r|w_1-w_2|+|w_1-w_2|,
\end{aligned}
\end{equation}
and hence
\begin{equation}\label{win}
 \sup_{r\ge1}r^4|G[u_1]w_1-G[u_1]w_2|\lesssim ||w_1-w_2||_{X_{r_0}}.
\end{equation}
For $r_0\le r\le1$, we have
$$
(r\partial_r+1)|2u_1+w_1+w_2|\lesssim r^{-\frac{5}{2}},
$$ and hence
\begin{equation}\nonumber
\begin{aligned}
&r\partial_r[(2u_1+w_1+w_2)(w_1-w_2)]\\
&=[r\partial_r(2u_1+w_1+w_2)](w_1-w_2)+[r\partial_r(w_1-w_2)](2u_1+w_1+w_2)\\
&\lesssim r^{-\frac{5}{2}}(|w_1-w_2|+r|\partial_r(w_1-w_2)|).
\end{aligned}
\end{equation}
Then it follows that
\begin{equation}\label{win1}
\begin{aligned}
&\int_{r_0}^1 |G[u_1]w_1-G[u_1]w_2|s^{\frac{7}{2}}ds\\
&\lesssim\int_{r_0}^1\left\{|s\partial_{s}[(2u_1+w_1+w_2)(w_1-w_2)]|+6|(2u_1+w_1+w_2)(w_1-w_2)|\right\}s^{\frac{7}{2}}ds\\
&\lesssim\int_{r_0}^1\left\{s^{-\frac{5}{2}}(s|\partial_s(w_1-w_2)|+|w_1-w_2|)\right\}s^{\frac{7}{2}}ds\\
&\lesssim\sup_{r_0\le r\le1}(r^{\frac{5}{2}}|w_1-w_2|+r^{\frac{7}{2}}|\partial_r(w_1-w_2)|)\int_{r_0}^1s^{-\frac{3}{2}}ds\lesssim r_0^{-\frac{1}{2}}||w_1-w_2||_{X_{r_0}}.
\end{aligned}
\end{equation}
Combining \eqref{win} and \eqref{win1}, this conclude the proof of \eqref{fixed-point1}.
\end{proof}

\subsection{Interior profiles}\label{INTERIOR}
The purpose of this subsection is to construct a radial solution of \eqref{swin} on $[0,r_0]$, where $0<r_0\ll1$ is given in Proposition \ref{taxi}.
We define
\begin{equation}\label{ALM}
    \bar{Q}(r)=\frac{1}{2r^3}\int_0^rQ(s)s^2ds.
\end{equation}
 By \eqref{steady-equation11},  $\bar{Q}$ satisfies
\begin{equation}\label{steady-equation}
\left\{\begin{aligned}
&\partial_{rr} \bar{Q}+\frac{4}{r}\partial_r\bar{Q}+6\bar{Q}^2+r\partial_r(\bar{Q}^2)=0, \\
&\bar{Q}(0)=\frac{1}{6},\ \ \bar{Q}'(0)=0.
\end{aligned} \right.
\end{equation}
We  define the linearized operators of  \eqref{steady-equation} at $\Phi_*$ and $\bar{Q}$, respectively,  by the following expressions:
\begin{equation}\label{Operators}
 H_\infty:=-\partial_{rr}-\frac{4}{r}\partial_r-12\Phi_*-2r\partial_r(\Phi_*\cdot),\
H:=-\partial_{rr}-\frac{4}{r}\partial_r-12\bar{Q}-2r\partial_r(\bar{Q}\cdot).
\end{equation}
We define $Y$ as the space of continuous functions on $[1,+\infty)$ such that the following norm is finite
$$
||w||_Y=\sup_{r\ge1}(r^3|w|+r^4|\partial_rw|).
$$
\begin{lemma}
The equation
$$
H_\infty(\phi)=0,\ \ \text{on}\ (0,+\infty),
$$
has two fundamental solutions
\begin{equation}\label{basis}
    \phi_1=\frac{\sin(\frac{\sqrt{7}}{2}\log(r))}{r^{\frac{5}{2}}},\ \ \phi_2=\frac{\cos(\frac{\sqrt{7}}{2}\log(r))}{r^{\frac{5}{2}}}.
\end{equation}
In addition, the inverse
\begin{equation}\label{zj1}
\psi(f)=\phi_1\int_r^{+\infty}f\phi_2\frac{2s^6}{\sqrt{7}}ds-\phi_2\int_r^{+\infty}f\phi_1\frac{2s^6}{\sqrt{7}}ds
\end{equation}
satisfies
$
H_\infty(\psi(f))=f
$
and
\begin{equation}\label{basis1}
||\psi(f)||_Y\lesssim\sup_{r\ge 1}r^{5}|f|.
\end{equation}
\end{lemma}
\begin{proof}
Let $\phi=r^k$, by $\Phi_*=\frac{1}{r^2}$, we have
$$
H_\infty(\phi)=-r^{k-2}(k^2+5k+8).
$$
Since the polynomial
$
k^2+5k+8=0
$
has two complex roots
$
k=\frac{-5\pm\sqrt{7}i}{2},
$
 the equation
$
H_\infty(\phi)=0
$
admits two explicit fundamental solutions
\begin{equation}\label{awc}
    \phi_1=\frac{\sin(\frac{\sqrt{7}}{2}\log(r))}{r^{\frac{5}{2}}},\ \ \phi_2=\frac{\cos(\frac{\sqrt{7}}{2}\log(r))}{r^{\frac{5}{2}}},
\end{equation}
and the corresponding Wronskian is given by
$
W(r)=\phi_1'\phi_2-\phi_2'\phi_1=\frac{\sqrt{7}}{2r^6}.
$
By the variation of constants, the solutions of equation
$
H_\infty(u)=f
$
are given by
\begin{equation}\label{sa}
u=\left(a_{1,0}+\int_r^{+\infty}f\phi_2\frac{2s^6}{\sqrt{7}}ds\right)\phi_1+
\left(a_{2,0}-\int_r^{+\infty}f\phi_1\frac{2s^6}{\sqrt{7}}ds\right)\phi_2,\ \ a_{1,0},\  a_{2,0}\in\mathbb{R}.
\end{equation}
Hence $$\psi(f)=\phi_1\int_r^{+\infty}f\phi_2\frac{2s^6}{\sqrt{7}}ds-\phi_2\int_r^{+\infty}f\phi_1\frac{2s^6}{\sqrt{7}}ds$$ satisfies $H_\infty(\psi(f))=f$ by choosing $a_{1,0}=a_{2,0}=0$ in \eqref{sa}.
For $r\ge1$, from \eqref{awc}, we have
\begin{equation}\label{so1}
\begin{aligned}
r^{3}|\psi(f)|&=r^{3}\left|\left(\int_r^{+\infty}f\phi_2\frac{2s^6}{\sqrt{7}}ds\right)\phi_1-\left(\int_r^{+\infty}f\phi_1\frac{2s^6}{\sqrt{7}}ds\right)\phi_2\right|\\
&\lesssim r^{\frac{1}{2}}\left(\int_r^{+\infty}|f|{s^{\frac{7}{2}}}ds\right)\lesssim \left(r^{\frac{1}{2}}\int_r^{+\infty}s^{-\frac{3}{2}}ds\right)\sup_{r\ge 1}r^{5}|f|\lesssim \sup_{r\ge 1}r^{5}|f|,
\end{aligned}
\end{equation}
and
\begin{equation}\label{sow}
\begin{aligned}
r^{4}|\partial_r\psi(f)|
&=r^{4}\left|\left(\int_r^{+\infty}f\phi_2\frac{2s^6}{\sqrt{7}}ds\right)\partial_r\phi_1-\left(\int_r^{+\infty}f\phi_1\frac{2s^6}{\sqrt{7}}ds\right)\partial_r\phi_2\right|\\
&\lesssim r^{\frac{1}{2}}\left(\int_r^{+\infty}|f|{s^{\frac{7}{2}}}ds\right)\lesssim \left(r^{\frac{1}{2}}\int_r^{+\infty}s^{-\frac{3}{2}}ds\right)\sup_{r\ge 1}r^{5}|f|\lesssim \sup_{r\ge 1}r^{5}|f|.
\end{aligned}
\end{equation}
We conclude the proof of \eqref{basis1} by \eqref{so1} and \eqref{sow}.
\end{proof}

\begin{lemma}
The asymptotic profile of $\bar{Q}$ as $r\to+\infty$ is
\begin{equation}\label{steady}
 \bar{Q}(r)=\Phi_*+\frac{c_5\sin(\frac{\sqrt{7}}{2}\log(r)+c_6)}{r^{\frac{5}{2}}}+O\left({r^{-3}}\right),
\end{equation}
where $c_5\ne0$ and $c_6\in\mathbb{R}$.
\end{lemma}
\begin{proof}
Assume that
\begin{equation}\label{steady00o}
\bar{Q}=\Phi_*+\varepsilon v
\end{equation}
solves \eqref{steady-equation} on $[1,\infty)$. Then $v$ satisfies
$
H_\infty(v)=\varepsilon(6v^2+r\partial_rv^2).
$
Let $v=\phi_1+w$, by $H_\infty(\phi_1)=0$, we have
$
H_\infty(w)=\varepsilon(6(\phi_1+w)^2+r\partial_r(\phi_1+w)^2).
$
We define $$G[\phi_1](w)=6(\phi_1+w)^2+r\partial_r(\phi_1+w)^2.$$
Next, we look for the solution of
\begin{equation}\label{so}
 w=\varepsilon\psi(G[\phi_1](w)),
\end{equation}
where $\psi(f)$ is defined in \eqref{zj1}. We claim that, if $w\in Y$, then
\begin{equation}\label{con}
    \sup_{r\ge1}r^5|G[\phi_1](w)|\lesssim1,
\end{equation}
and for $w_1$, $w_2\in Y$, it holds that
\begin{equation}\label{con1}
  \sup_{r\ge1}r^5|G[\phi_1](w_1)-G[\phi_1](w_2)|\lesssim ||w_1-w_2||_{Y}.
\end{equation}
If the above claim holds, for $\varepsilon>0$ small enough, by the resolvent estimate \eqref{basis1}  and the Banach fixed point theorem, there exists a unique solution $w\in Y$ to \eqref{so} and hence we find a $v$ for \eqref{steady00o}. Finally we get \eqref{steady} by \eqref{steady00o}.

It remains to show estimates \eqref{con} and \eqref{con1}.
By \eqref{basis} and  the definition of the space $Y$,
for $r\ge1$  and $w\in Y$, we have
\begin{equation}\nonumber
\begin{aligned}
r^{5}|G[\phi_1](w)|&=r^5\{6(\phi_1+w)^2+r\partial_r(\phi_1+w)^2\}\\
&\lesssim r^{5}[(\phi_1+w+2r\partial_r(\phi_1+w))(\phi_1+w)]\\
&\lesssim r^{5}(r^{-5}+r^{-6}+r^{-\frac{11}{2}})\lesssim 1.
\end{aligned}
\end{equation}
For $r\ge1$ and $w_i\in Y$ $(i=1,2)$, by \eqref{basis}   and the definition of the space $Y$, we get
$$
|w_1+w_2+2\phi_1|\lesssim r^{-\frac{5}{2}},\ \ |r\partial_r(w_1+w_2+2\phi_1)|\lesssim r^{-\frac{5}{2}}.
$$
Hence we have
\begin{equation}\nonumber
\begin{aligned}
&|G[\phi_1](w_1)-G[\phi_1](w_2)|\\
&=|6(w_1+w_2+2\phi_1)(w_1-w_2)+r\partial_r[(w_1+w_2+2\phi_1)(w_1-w_2)]|\\
&\lesssim r^{-\frac{5}{2}}|w_1-w_2|+|r\partial_r(w_1+w_2+2\phi_1)||w_1-w_2|+|r\partial_r(w_1-w_2)||w_1+w_2+2\phi_1|\\
&\lesssim r^{-\frac{5}{2}}(|w_1-w_2|+|r\partial_r(w_1-w_2)|),
\end{aligned}
\end{equation}
and
\begin{equation}\nonumber
\begin{aligned}
r^{5}|G[\phi_1](w_1)-G[\phi_1](w_2)|&\lesssim r^5(r^{-\frac{5}{2}}|w_1-w_2|+r^{-\frac{5}{2}}|r\partial_r(w_1-w_2)|)\\
&=r^{-\frac{1}{2}}(r^3|w_1-w_2|+r^{4}|\partial_r(w_1-w_2)|)\\
&\le||w_1-w_2||_Y.
\end{aligned}
\end{equation}
This completes the proof of \eqref{con} and \eqref{con1}.
\end{proof}

Let $r_1\gg1$. We define $Y_{r_1}$ as the space of continuous functions on $[0,r_1]$ in which the following norm is finite:
\begin{equation}\label{jsj}
  ||w||_{Y_{r_1}}=\sup_{0\le r\le r_1}(1+r)^{-\frac{1}{2}}(|w|+|r\partial_r w|).
\end{equation}

\begin{lemma}
Let $H$ be defined in \eqref{Operators}. Then the following results hold.
\\
1. The basis of the fundamental solutions:    There holds
$$
H(\Lambda \bar{Q})=0,\ \ H(\rho)=0
 $$
 with the following asymptotic behavior as $r\to+\infty$,
 $$
\Lambda \bar{Q}=\frac{c_7\sin(\frac{\sqrt{7}}{2}\log(r)+c_8)}{r^{\frac{5}{2}}}+O(r^{-3}),\ \
\rho=\frac{c_9\sin(\frac{\sqrt{7}}{2}\log(r)+c_{10})}{r^{\frac{5}{2}}}+O(r^{-3})
,$$
where $c_7, c_9\ne0$ and $c_8,c_{10}\in\mathbb{R}.$\\
2. The continuity of the resolvent: The inverse
$$
S(f)=\left(\int_0^rf\Lambda \bar{Q}\exp\left({\int2s\bar{Q}(s)ds}\right)s^4ds\right)\rho-
\left(\int_0^rf\rho\exp\left({\int2s\bar{Q}(s)ds}\right)s^4ds\right)\Lambda \bar{Q}
,
$$
satisfies
$
H(S(f))=f
$
and
\begin{equation}\label{resolvent}
  ||S(f)||_{Y_{r_1}}\lesssim \sup_{0\le r\le r_1}(1+r)^2|f|.
\end{equation}
\end{lemma}
\begin{proof}
\textbf{Step 1.} Fundamental solutions. Let
$$
\bar{Q}_\lambda(r)=\lambda^2\bar{Q}(\lambda r), \ \ \lambda>0.
$$
Then
$$
\partial_{rr} \bar{Q}_\lambda+\frac{4}{r}\partial_r\bar{Q}_\lambda+6\bar{Q}_\lambda^2+r\partial_r(\bar{Q}_\lambda^2)=0,\ \  \lambda>0.
$$
Differentiating the above equation with $\lambda$ and evaluating at $\lambda=1$ yields
$
H(\Lambda \bar{Q})=0.
$
Let $\rho$ be another solution to $H(\rho)=0$ which is linearly independent of $\Lambda \bar{Q}$.
We claim that, all solutions of $H(\phi)=0$ admit an expansion
\begin{equation}\label{expansion}
    \phi=a_{1,0}\phi_1+a_{2,0}\phi_2+O(r^{-3}), \ \ \text{as}\ \ r\to+\infty,
\end{equation}
where $a_{1,0}$, $a_{2,0}\in\mathbb{R}$ and $\phi_1,$ $\phi_2$ are defined in \eqref{basis}.

We rewrite $H(\phi)=0$ in the following form
\begin{equation}\label{sz}
 H_\infty(\phi)= -\partial_{rr} \phi-\frac{4}{r}\partial_r\phi-12\Phi_*\phi-2r\partial_r(\Phi_*\phi)=f,
\end{equation}
where
$$
f=f(\phi)=12(\bar{Q}-\Phi_*)\phi+2r\partial_r((\bar{Q}-\Phi_*)\phi).
$$
Next, we look for the solution of equation \eqref{sz}. By \eqref{sa}, we shall find a solution in a form
\begin{equation}\label{banach}
\phi=a_{1,0}\phi_1+a_{2,0}\phi_2+\widetilde{\phi},
\end{equation}
where
\begin{equation}\nonumber
\widetilde{\phi}=F(\widetilde{\phi})=\left(\int_r^{+\infty}f({\phi})\phi_2\frac{2s^6}{\sqrt{7}}ds\right)\phi_1-
\left(\int_r^{+\infty}f({\phi})\phi_1\frac{2s^6}{\sqrt{7}}ds\right)\phi_2:=F_1(\widetilde{\phi})-F_2(\widetilde{\phi}).
\end{equation}
It follows from \eqref{basis} that
\begin{equation}\label{as}
    |r\partial_r(\phi_1+\phi_2)|\lesssim r^{-\frac{5}{2}}.
\end{equation}
Recall from \eqref{steady} that
\begin{equation}\label{as1}
|\bar{Q}-\Phi_*|\lesssim r^{-\frac{5}{2}},\ \ |r\partial_r(\bar{Q}-\Phi_*)|\lesssim r^{-\frac{5}{2}},\ \ \text{for}\ r\ge1.
\end{equation}
For $r\ge1$, by \eqref{as} and \eqref{as1},  we have
\begin{equation}\nonumber
\begin{aligned}
F_1(\widetilde{\phi})&\lesssim\left(\int_r^{+\infty}12|\bar{Q}-\Phi_*||a_{1,0}\phi_1+a_{2,0}\phi_2+\widetilde{\phi}|\frac{2s^6|\phi_2|}{\sqrt{7}}ds\right)|\phi_1|\\
&\ \ \ +\left(\int_r^{+\infty}2|r\partial_r(\bar{Q}-\Phi_*)||a_{1,0}\phi_1+a_{2,0}\phi_2+\widetilde{\phi}|\frac{2s^6|\phi_2|}{\sqrt{7}}ds\right)|\phi_1|\\
&\ \ \ +\left(\int_r^{+\infty}2|\bar{Q}-\Phi_*||r\partial_r(a_{1,0}\phi_1+a_{2,0}\phi_2+\widetilde{\phi})|\frac{2s^6|\phi_2|}{\sqrt{7}}ds\right)|\phi_1|\\
&\lesssim r^{-\frac{5}{2}}\left(\int_r^{+\infty}
s^{-\frac{3}{2}}+s|\widetilde{\phi}|
ds\right)+r^{-\frac{5}{2}}\left(\int_r^{+\infty}
s|r\partial_r\widetilde{\phi}|
ds\right)\\
&\le r^{-3}+r^{-\frac{5}{2}}\left(\int_r^{+\infty}
s(|\widetilde{\phi}|+|r\partial_r\widetilde{\phi}|)ds\right).
\end{aligned}
\end{equation}
Similarly,
\begin{equation}\nonumber
\begin{aligned}
F_2(\widetilde{\phi})\lesssim r^{-3}+r^{-\frac{5}{2}}\left(\int_r^{+\infty}
s(|\widetilde{\phi}|+|r\partial_r\widetilde{\phi}|)ds\right).
\end{aligned}
\end{equation}
Hence
\begin{equation}\label{ban1}
\begin{aligned}
F(\widetilde{\phi})\lesssim r^{-3}+r^{-\frac{5}{2}}\left(\int_r^{+\infty}
s(|\widetilde{\phi}|+|r\partial_r\widetilde{\phi}|)ds\right)
\end{aligned}
\end{equation}
and
\begin{equation}\label{ban2}
\begin{aligned}
F(\widetilde{\phi}_1)-F(\widetilde{\phi}_2)&\lesssim r^{-\frac{5}{2}}\left(\int_r^{+\infty}
s(|\widetilde{\phi}_1-\widetilde{\phi}_2|+|r\partial_r(\widetilde{\phi}_1-\widetilde{\phi}_2)|)ds\right).
\end{aligned}
\end{equation}
In the same manner, we have
\begin{equation}\label{ban3}
\begin{aligned}
r\partial_rF(\widetilde{\phi})&\lesssim r^{-3}+r^{-\frac{5}{2}}\left(\int_r^{+\infty}
s(|\widetilde{\phi}|+|r\partial_r\widetilde{\phi}|)ds\right),
\end{aligned}
\end{equation}
and
\begin{equation}\label{ban4}
\begin{aligned}
r\partial_r(F(\widetilde{\phi}_1)-F(\widetilde{\phi}_2))&\le r^{-\frac{5}{2}}\left(\int_r^{+\infty}
s(|\widetilde{\phi}_1-\widetilde{\phi}_2|+|r\partial_r(\widetilde{\phi}_1-\widetilde{\phi}_2)|)ds\right).
\end{aligned}
\end{equation}

For $R\gg1$, we define $Z$ as the space of continuous functions on $[R,+\infty)$ such that the following norm is finite
$$
||\phi||_Z=\sup_{r\ge R}r^{3}(|{\phi}|+|r\partial_r{\phi}|).
$$

By \eqref{ban1}-\eqref{ban4}  and the Banach fixed point theorem, there exists a unique solution $\widetilde{\phi}$ that satisfies $F(\widetilde{\phi})=\widetilde{\phi}$ with the bound
$
||\widetilde{\phi}||_Z\lesssim1,
$
and hence we find a solution $\phi$ in the form \eqref{banach} that solves \eqref{sz}. This
 proves \eqref{expansion}.

Since $H(\Lambda \bar{Q})=H(\rho)=0$, by \eqref{basis} and \eqref{expansion},  we have
\begin{equation}\label{asymp}
   \Lambda \bar{Q}=\frac{c_7\sin(\frac{\sqrt{5}}{2}\log(r)+c_8)}{r^{\frac{5}{2}}}+O(r^{-3}),\ \
\rho=\frac{c_9\sin(\frac{\sqrt{5}}{2}\log(r)+c_{10})}{r^{\frac{5}{2}}}+O(r^{-3})
,\ \ r\to\infty,
\end{equation}
where $c_7, c_9\ne0$ and $c_8,c_{10}\in\mathbb{R}.$
\\
\textbf{Step 2.} The estimate of the resolvent. We compute the Wronskian
$$
W=\Lambda \bar{Q}'\rho-\Lambda \bar{Q}\rho',\ W'=-\left(\frac{4}{r}+2r\bar{Q}\right)W,\
W=\frac{\exp({-\int2r\bar{Q}dr})}{r^4}.
$$
Take $R_0>0$ small enough. By the definition of $W$, we have
$\frac{W}{(\Lambda \bar{Q})^2}=-\frac{d}{dr}\left(\frac{\rho}{\Lambda \bar{Q}}\right),$
then integrating over $[r,R_0]$ yields
\begin{equation}\label{cf}
   \rho(r)=\Lambda \bar{Q}(r)\int_r^{R_0}\frac{\exp({-\int2s\bar{Q}ds})}{s^4(\Lambda \bar{Q})^2}ds+\frac{\Lambda \bar{Q}(r)\rho(R_0)}{\Lambda \bar{Q}(R_0)}.
\end{equation}
By $\bar{Q}(0)=\frac{1}{6}$ and $\bar{Q}'(0)=0$, we have
\begin{equation}\label{beh}
    |\bar{Q}|+|r\partial_r \bar{Q}|\lesssim1, \ r\in[0,1].
\end{equation}
Then by \eqref{cf}, one has
\begin{equation}\label{beah}
|\rho(r)|\lesssim\frac{1}{r^3},\ |\partial_r\rho(r)|\lesssim\frac{1}{r^4},\ \ \text{as}\ r\to 0.
\end{equation}
If
$
H(w)=f,
$
then by the variation of constants, one obtain
\begin{equation}\label{asd}
    w=\left(a_3+\int_0^r\frac{f\Lambda \bar{Q}}{W}\right)\rho+
\left(a_4-\int_0^r\frac{f\rho}{W}\right)\Lambda \bar{Q},\ \ a_3,\ a_4\in\mathbb{R}.
\end{equation}
Hence, $$S(f)=\rho\int_0^r\frac{f\Lambda \bar{Q}}{W}ds-\Lambda \bar{Q}\int_0^r\frac{f\rho}{W}ds$$ satisfies
$
H(S(f))=f
$
by choosing $a_3=a_4=0$ in \eqref{asd}. For $0\le r\le1$, by \eqref{beh} and \eqref{beah}, we get the estimate
\begin{equation}\label{can}
\begin{aligned}
&(1+r)^{-\frac{1}{2}}|S(f)|\\
&=(1+r)^{-\frac{1}{2}}\left|\left(\int_0^rf\Lambda \bar{Q}\exp\left({\int2s\bar{Q}ds}\right)s^4ds\right)\rho-
\left(\int_0^rf\rho\exp\left({\int2s\bar{Q}ds}\right)s^4ds\right)\Lambda \bar{Q}\right|\\
&\lesssim\left(\frac{1}{r^3}\int_0^rs^4ds+\int_0^rsds\right)\sup_{0\le r\le1}|f|\lesssim\sup_{0\le r\le r_1}(1+r)^2|f|.
\end{aligned}
\end{equation}
 For $1\le r\le r_1$, we know from \eqref{steady} that
$$
|\bar{Q}(r)|\lesssim\frac{1}{r^2},\ \ \exp\left({\int2s\bar{Q}(s)ds}\right)\lesssim r^2.
$$
Then combining \eqref{asymp} and \eqref{can}, we get
\begin{equation}\label{can1}
\begin{aligned}
&(1+r)^{-\frac{1}{2}}|S(f)|\\
&\lesssim(1+r)^{-\frac{1}{2}}\left|\left(\int_0^1f\rho\exp\left({\int2s\bar{Q}ds}\right)s^4ds\right)\Lambda \bar{Q}-
\left(\int_0^1f\Lambda \bar{Q}\exp\left({\int2s\bar{Q}ds}\right)s^4ds\right)\rho\right|\\
&\ \ +(1+r)^{-\frac{1}{2}}\left|\left(\int_1^rf\rho\exp\left({\int2s\bar{Q}ds}\right)s^4ds\right)\Lambda \bar{Q}-
\left(\int_1^rf\Lambda \bar{Q}\exp\left({\int2s\bar{Q}ds}\right)s^4ds\right)\rho\right|\\
&\lesssim \sup_{0\le r\le r_1}(1+r)^2|f|+r^{-3}\int^r_1|f|s^{\frac{7}{2}}ds
\lesssim\sup_{0\le r\le r_1}(1+r)^2|f|.
\end{aligned}
\end{equation}
Similarly, for $0\le r\le r_1$, we also have
\begin{equation}\label{can2}
\begin{aligned}
(1+r)^{-\frac{1}{2}}|r\partial_rS(f)|\lesssim \sup_{0\le r\le r_1}(1+r)^2|f|.
\end{aligned}
\end{equation}
We finally get \eqref{resolvent} by \eqref{can}, \eqref{can1}, and  \eqref{can2}.
\end{proof}
We are now in the position to construct a  interior solutions for the equation \eqref{swin}.
\begin{proposition}\label{interior}
    Let $0<r_0\ll1$ and $0<\lambda\le r_0$. There exists a radial solution $u$ to
\begin{equation}\label{rai}
   \Delta \Phi-\frac{1}{2}\Lambda \Phi+6\Phi^2+y\cdot \nabla( \Phi^2)=0,\ \ 0\le r\le r_0,
\end{equation}
with the form
 $$
\Phi=\frac{1}{\lambda^2}(\bar{Q}+\lambda^4 Q_1)\left(\frac{r}{\lambda}\right)
$$
with
$||Q_1||_{Y_{\frac{r_0}{\lambda}}}\lesssim1.$
\end{proposition}
\begin{proof}
\textbf{Step 1.} Application of the Banach fixed-point theorem. We look for $\Phi$ of the form
$$
     \Phi=\frac{1}{\lambda^2}(\bar{Q}+\lambda^4 Q_1)\left(\frac{r}{\lambda}\right),
     $$
so that $\Phi$ solves \eqref{rai} on $[0, r_0]$. Then,
\begin{equation}\label{ac}
   H(Q_1)=J[\bar{Q},\lambda]Q_1,\ \ 0\le r\le r_1,
\end{equation}
where $r_1=\frac{r_0}{\lambda}\ge1$ such that
$
\lambda^2r_1^2=r_0^2\ll1,
$
 and
$$
J[\bar{Q},\lambda]Q_1=-\frac{1}{2\lambda^2}\Lambda \bar{Q}-\frac{1}{2}\lambda^2\Lambda Q_1+\lambda^4(6Q_1^2+r\partial_r(Q_1^2)).
$$
For $w\in Y_{r_1}$, we claim the following estimates:
\begin{equation}\label{es1}
   \sup_{0\le r\le r_1} (1+r)^2|J[\bar{Q},\lambda]w|\lesssim1,
\end{equation}
and
\begin{equation}\label{es2}
    \sup_{0\le r\le r_1} (1+r)^2|J[\bar{Q},\lambda]w_1-J[\bar{Q},\lambda]w_2|\lesssim \lambda^2 r_1^2||w_1-w_2||_{Y_{r_1}}.
\end{equation}
If \eqref{es1} and \eqref{es2} hold, by $\lambda^2r_1^2\ll1$, the resolvent estimate \eqref{resolvent}, and the Banach fixed point theorem, there exists a unique solution $Q_1$ of \eqref{ac} with
$
||Q_1||_{Y_{\frac{r_0}{\lambda}}}\lesssim1.
$
\\
\textbf{Step 2.} Proof of estimates \eqref{es1} and \eqref{es2}.
For $0\le r\le r_1$ and $w\in Y_{r_1}$, by the definition of the space $Y_{r_1}$ in \eqref{jsj}, we have
$
|\Lambda w|\lesssim1.
$
Then, by $|\Lambda \bar{Q}|\lesssim1$, we get
\begin{equation}\nonumber
\begin{aligned}
(1+r)^2|J[\bar{Q},\lambda]w|\lesssim1,\ \text{on}\ [0,r_1],
\end{aligned}
\end{equation}
which concludes the proof of \eqref{es1}.

For $0\le r\le r_1$ and $w_1,w_2\in Y_{r_1}$, we have
$$
|\Lambda (w_1-w_2)|\lesssim ||w_1-w_2||_{Y_{r_1}}, \ |w_1+w_2|\lesssim\ r\partial_r(w_1+w_2)\lesssim1.
$$
Then it follows that
\begin{equation}\nonumber
\begin{aligned}
r\partial_r[(w_1+w_2)(w_1-w_2)]&=(w_1-w_2)r\partial_r(w_1+w_2)+(w_1+w_2)r\partial_r(w_1-w_1)\\
&\lesssim |w_1-w_2|+|r\partial_r(w_1-w_2)|\le ||w_1-w_2||_{Y_{r_1}}.
\end{aligned}
\end{equation}
Hence,
\begin{equation}\nonumber
\begin{aligned}
(1+r)^2|J[\bar{Q},\lambda]w_1-J[\bar{Q},\lambda]w_2|&\lesssim \lambda^2(1+r)^2|\Lambda (w_1-w_2)|+
\lambda^4(1+r)^2(w_1+w_2)(w_1-w_2)\\
&\ \ \ +\lambda^4(1+r)^2r\partial_r[(w_1+w_2)(w_1-w_1)]\\
&\lesssim \lambda^2(1+r)^2||w_1-w_2||_{Y_{r_1}}\lesssim \lambda^2r_1^2||w_1-w_2||_{Y_{r_1}},
\end{aligned}
\end{equation}
which concludes the proof of \eqref{es2}.
\end{proof}
\subsection{The matching at $r=r_0$}
In this subsection, we prove Proposition \ref{swin1} by matching the value of the exterior solution and interior solution at $r=r_0$ up to the first-order derivative.
\begin{proof}[Proof of Proposition \ref{swin1}]
The proof is divided into six steps.\\
\textbf{Step 1.} Initial setting. From \eqref{origin}, we have
$$
u_1=\frac{c_1\sin(\frac{\sqrt{7}}{2}\log(r)+c_2)}{r^{\frac{5}{2}}}+O(r^{-\frac{1}{2}})\ \ \text{as}\ r\to 0,\ c_1\ne0,\ c_2\in \mathbb{R},
$$
then
$$
\Lambda u_1=c_1\frac{-\frac{1}{2}\sin(\frac{\sqrt{7}}{2}\log(r)+c_2)+\frac{\sqrt{7}}{2}\cos(\frac{\sqrt{7}}{2}\log(r)+c_2)}{r^{\frac{5}{2}}}+O(r^{-\frac{1}{2}})\ \ \text{as}\ r\to 0.
$$
We choose $0< r_0\ll1$ such that
\begin{equation}\label{nonzero}
  u_1(r_0)=\frac{c_1}{r_0^{\frac{5}{2}}}+O(r_0^{-\frac{1}{2}}),\ \
\Lambda u_1(r_0)=-\frac{c_1}{2r_0^{\frac{5}{2}}}+O(r_0^{-\frac{1}{2}}).
\end{equation}
Then, we choose $\varepsilon$ and $\lambda$ satisfying
\begin{equation}\label{bef}
0<\varepsilon\ll r_0^{\frac{1}{2}},\ \ 0<\lambda\le r_0.
\end{equation}
By Proposition \ref{taxi}, there exists an radial exterior solution $\Phi_{\text{ext}}[\varepsilon]$ satisfying
$$
\Delta \Phi_{\text{ext}}-\frac{1}{2}\Lambda \Phi_{\text{ext}}+6\Phi_{\text{ext}}^2+y\cdot \nabla( \Phi_{\text{ext}}^2)=0, \ \ r\ge r_0
$$
with the form
\begin{equation}\label{lj}
\Phi_{\text{ext}}[\varepsilon]=\Phi_*+\varepsilon u_1+\varepsilon w
\end{equation}
and
\begin{equation}\label{hom}
||w||_{X_{r_0}}\lesssim\varepsilon r_0^{-\frac{1}{2}}.
\end{equation}
By Proposition \ref{interior}, there exists an radial interior solution $\Phi_{\text{int}}[\lambda]$
satisfying
$$
\Delta \Phi_{\text{int}}-\frac{1}{2}\Lambda \Phi_{\text{int}}+6\Phi_{\text{int}}^2+y\cdot \nabla( \Phi_{\text{int}}^2)=0, \ \ 0\le r\le  r_0
$$
with the form
\begin{equation}\label{liuq}
\Phi_{\text{int}}[\lambda](r)=\frac{1}{\lambda^2}(\bar{Q}+\lambda^4 Q_1)\left(\frac{r}{\lambda}\right),
\end{equation}
with
\begin{equation}\label{hom1}
 ||Q_1||_{Y_{\frac{r_0}{\lambda}}}\lesssim1.
\end{equation}
Next, we need to match the values of $\Phi_{\text{ext}}$ with $\Phi_{\text{int}}$, and $\Phi'_{\text{ext}}$ with $\Phi'_{\text{int}}$ respectively at $r=r_0$, that is,
$$
\Phi_{\text{ext}}[\varepsilon](r_0)=\Phi_{\text{int}}[\lambda](r_0),\ \ \Phi'_{\text{ext}}[\varepsilon](r_0)=\Phi'_{\text{int}}[\lambda](r_0).
$$
\textbf{Step 2.}  The matching of $\Phi_{\text{ext}}$ with $\Phi_{\text{int}}$ at $r=r_0$. We introduce the map
$$
F[r_0](\varepsilon,\lambda)=\Phi_{\text{ext}}[\varepsilon](r_0)-\Phi_{\text{int}}[\lambda](r_0).
$$
We compute
$$
\partial_\varepsilon F[r_0](\varepsilon,\lambda)=
\partial_\varepsilon \Phi_{\text{ext}}[\varepsilon](r_0)=
u_1(r_0)+w(r_0)+\varepsilon \partial_\varepsilon w(r_0).
$$
By \eqref{swi} and \eqref{nonzero}, we have
\begin{equation}\label{j1}
    \partial_\varepsilon F[r_0](0,0)=u_1(r_0)\ne0.
\end{equation}
For $\lambda\to 0_+$, from the asymptotic behavior of $\bar{Q}$ in \eqref{steady} and the definition of the space $Y_{r_1}$ in  \eqref{jsj}, combining \eqref{hom1},  we have
\begin{equation}\nonumber
\begin{aligned}
\left|\frac{1}{\lambda^2}(\bar{Q}-\Phi_*+\lambda^4 Q_1)\left(\frac{r_0}{\lambda}\right)\right|\lesssim \left|\frac{1}{\lambda^2}\left(r^{-\frac{5}{2}}+\lambda^4{(1+r)^{\frac{1}{2}}}\right)\left(\frac{r_0}{\lambda}\right)\right|
=\lambda^{\frac{1}{2}}\left[r_0^{-\frac{5}{2}}+\lambda(\lambda+r_0)^{\frac{1}{2}}\right].
\end{aligned}
\end{equation}
Hence
$$
\lim_{\lambda\to0_+}\left|\frac{1}{\lambda^2}(\bar{Q}-\Phi_*+\lambda^4 Q_1)\left(\frac{r_0}{\lambda}\right)\right|=0.
$$
Combining $\Phi_*(r)=\frac{1}{\lambda^2}\Phi_*(\frac{r}{\lambda})$, we have
\begin{equation}\label{j2}
F[r_0](0,0)=\Phi_*(r_0)-\Phi_*(r_0)=0.
\end{equation}
Combining \eqref{j1} and \eqref{j2}, by the implicit function theorem, there exists $0<\lambda_0\le r_0$ and a continuous function $\varepsilon(\lambda)$  defined on $[0,\lambda_0)$ such that $\varepsilon(0)=0$ and
\begin{equation}\label{exterios}
F[r_0](\varepsilon(\lambda),\lambda)=0\ \ \text{for}\  \lambda\in[0,\lambda_0),
\end{equation}
i.e.,
$$
    \Phi_{\text{ext}}[\varepsilon(\lambda)](r_0)=\Phi_{\text{int}}[\lambda](r_0)\ \ \text{for}\
\lambda\in[0,\lambda_0).
$$
\textbf{Step 3.} Estimate of $\varepsilon(\lambda)$. We claim that for $\lambda\in[0,\lambda_0)$, there holds that
\begin{equation}\label{dq}
   \varepsilon(\lambda)=\frac{1}{u_1(r_0)\lambda^2}(\bar{Q}-\Phi_*)\left(\frac{r_0}{\lambda}\right)+O(\lambda(\lambda^{\frac{1}{2}}r_0^{3}+r_0^{-\frac{1}{2}})).
\end{equation}
In fact, since
$$
\Phi_{\text{ext}}[\varepsilon(\lambda)](r_0)=\Phi_{\text{int}}[\lambda](r_0)\ \ \text{for}\
\lambda\in[0,\lambda_0),$$
i.e.,
$$
\varepsilon(\lambda)u_1(r_0)+\varepsilon(\lambda)w(r_0)=\frac{1}{\lambda^2}(\bar{Q}-\Phi_*+\lambda^4 Q_1)\left(\frac{r_0}{\lambda}\right),\ \ \text{for}\
\lambda\in[0,\lambda_0).
$$
By \eqref{bef}, we know that
\begin{equation}\label{sq}
  |\varepsilon(\lambda)|\lesssim \lambda^{\frac{1}{2}}.
\end{equation}
Then by \eqref{origin}, \eqref{hom} and \eqref{hom1}, we have
\begin{equation}\nonumber
\begin{aligned}
\varepsilon(\lambda)&=\frac{1}{\lambda^2u_1(r_0)}(\bar{Q}-\Phi_*+\lambda^4Q_1)\left(\frac{r_0}{\lambda}\right)-\frac{\varepsilon(\lambda)w(r_0)}{u_1(r_0)}\\
&=\frac{1}{\lambda^2u_1(r_0)}(\bar{Q}-\Phi_*)\left(\frac{r_0}{\lambda}\right)
+O(\lambda(\lambda^{\frac{1}{2}}r_0^{3}+r_0^{-\frac{1}{2}})),
\end{aligned}
\end{equation}
which proves our claim.\\
\textbf{Step 4.} Computation of the spatial derivatives.
We consider the difference of the spatial derivatives at $r_0$
$$
\mathcal{F}[r_0](\lambda)=\Phi_{\text{ext}}[\varepsilon(\lambda)]'(r_0)-\Phi_{\text{int}}[\lambda]'(r_0),\ \ \lambda\in[0,\lambda_0).
$$
We claim that $
\mathcal{F}[r_0](\lambda)$ admits the following expansion
\begin{equation}\label{long2}
\begin{aligned}
\mathcal{F}[r_0](\lambda)&=\lambda^{\frac{1}{2}}\left\{
\frac{c_1c_7\sqrt{7}}{2u_1(r_0)r_0^6}\sin\left(-\frac{\sqrt{7}}{2}\log \lambda+c_8-c_2 \right)+O\left(\lambda^{\frac{1}{2}} r_0^{-\frac{1}{2}}\left(r_0^{-\frac{7}{2}}+\lambda^{\frac{3}{2}}\right)\right)\right\}.
\end{aligned}
\end{equation}
From \eqref{hom} and \eqref{sq}, it follows that
$$
|\varepsilon(\lambda)w'(r_0)|\lesssim \lambda^{\frac{1}{2}}|w'(r_0)|\lesssim
\lambda r_0^{-4}.
$$
From \eqref{hom1}, we get
$
\lambda^2|T'(\frac{r_0}{\lambda})|\lesssim \lambda^{\frac{5}{2}}r_0^{-\frac{1}{2}}.
$
By \eqref{dq}, we have
\begin{equation}\label{long1}
\begin{aligned}
&\mathcal{F}[r_0](\lambda)=
\varepsilon(\lambda)u_1'(r_0)-\frac{1}{\lambda^3}(\bar{Q}'-\Phi_*')\left(\frac{r_0}{\lambda}\right)+O\left(\lambda\left(r_0^{-4}+\lambda^{\frac{3}{2}}r_0^{-\frac{1}{2}}\right)\right)\\
&=\frac{1}{u_1(r_0)\lambda^2}(\bar{Q}-\Phi_*)\left(\frac{r_0}{\lambda}\right)u_1'(r_0)-\frac{1}{\lambda^3}(\bar{Q}'-\Phi_*')\left(\frac{r_0}{\lambda}\right)+O\left(\lambda\left(r_0^{-4}+\lambda^{\frac{3}{2}}r_0^{-\frac{1}{2}}\right)\right)\\
&=\frac{\lambda^{\frac{1}{2}}}{u_1(r_0)r_0^\frac{5}{2}}\left\{\left(\frac{r_0}{\lambda}\right)^{\frac{5}{2}}(\bar{Q}-\Phi_*)\left(\frac{r_0}{\lambda}\right)u_1'(r_0)-\left(\frac{r_0}{\lambda}\right)^{\frac{7}{2}}(\bar{Q}'-\Phi_*')\left(\frac{r_0}{\lambda}\right)\frac{u_1(r_0)}{r_0}\right\} \\
&\ \ \ \ +O\left(\lambda\left(r_0^{-4}+\lambda^{\frac{3}{2}}r_0^{-\frac{1}{2}}\right)\right).
\end{aligned}
\end{equation}
Recalling \eqref{origin} and \eqref{steady}, by simple calculations, one has
$$
u_1(r)=\frac{c_1\sin(\frac{\sqrt{7}}{2}\log(r)+c_2)}{r^{\frac{5}{2}}}+O(r^{-\frac{1}{2}})\ \ \text{as}\ r\to0,
$$
$$
u_1'(r)=\frac{-5c_1\sin(\frac{\sqrt{7}}{2}\log(r)+c_2)}{2r^{\frac{7}{2}}}+\frac{\sqrt{7}c_1\cos(\frac{\sqrt{7}}{2}\log(r)+c_2)}{2r^{\frac{7}{2}}}+O(r^{-\frac{3}{2}})\ \ \text{as}\ r\to0,
$$
$$
\bar{Q}(r)-\Phi_*(r)=\frac{c_7\sin(\frac{\sqrt{7}}{2}\log(r)+c_8)}{r^{\frac{5}{2}}}+O(r^{-3})\ \ \text{as}\ r\to+\infty,
$$
$$
\bar{Q}'(r)-\Phi_*'(r)=\frac{-5c_7\sin(\frac{\sqrt{7}}{2}\log(r)+c_8)}{2r^{\frac{7}{2}}}+\frac{\sqrt{7}c_7\cos(\frac{\sqrt{7}}{2}\log(r)+c_8)}{2r^{\frac{7}{2}}}+O(r^{-4})\ \ \text{as}\ r\to+\infty.
$$
Then it follows from the above results that
\begin{equation*}\label{long}
\begin{aligned}
&\left(\frac{r_0}{\lambda}\right)^{\frac{5}{2}}(\bar{Q}-\Phi_*)\left(\frac{r_0}{\lambda}\right)u_1'(r_0)-\left(\frac{r_0}{\lambda}\right)^{\frac{7}{2}}(\bar{Q}'-\Phi_*')\left(\frac{r_0}{\lambda}\right)\frac{u_1(r_0)}{r_0}\\
&=\frac{c_1c_7}{r_0^\frac{7}{2}}\sin\left(\frac{\sqrt{7}}{2}(\log r_0-\log\lambda)+c_8\right)\times\left(\frac{\sqrt{7}}{2}\cos\left(\frac{\sqrt{7}}{2}\log r_0+c_2\right)-\frac{5}{2}\sin\left(\frac{\sqrt{7}}{2}\log r_0+c_2\right)\right)\\
&\ \ \ -\frac{c_1c_7}{r_0^\frac{7}{2}}\left(\frac{\sqrt{7}}{2}\cos\left(\frac{\sqrt{7}}{2}(\log r_0-\log\lambda)+c_8\right)- \frac{5}{2}\sin\left(\frac{\sqrt{7}}{2}(\log r_0-\log\lambda)+c_8\right)\right)\\
&\ \ \ \times
\sin\left(\frac{\sqrt{7}}{2}\log(r_0)+c_2\right)+O\left(\lambda^{\frac{1}{2}}\left(r_0^{-4}+\lambda^{\frac{3}{2}}r_0^{-\frac{1}{2}}\right)\right)\\
&=\frac{c_1c_7\sqrt{7}}{2r_0^\frac{7}{2}}\sin\left(-\frac{\sqrt{7}}{2}\log \lambda+c_8-c_2 \right)+O\left(\lambda^{\frac{1}{2}}\left(r_0^{-4}+\lambda^{\frac{3}{2}}r_0^{-\frac{1}{2}}\right)\right).
\end{aligned}
\end{equation*}
Inserting the above identity into \eqref{long1}, we obtain \eqref{long2}. This proves our claim.
\\
\textbf{Step 5.} The matching of $\Phi_{\text{ext}}'$ with $\Phi_{\text{int}}'$ at $r=r_0$.
For $\delta_0>0$ small enough, we define
$$
\lambda_{k,+}=\exp\left(\frac{2(-k\pi+c_8-c_2-\delta_0)}{\sqrt{7}}\right),\ \
\lambda_{k,-}=\exp\left(\frac{2(-k\pi+c_8-c_2+\delta_0)}{\sqrt{7}}\right).
$$
Since
$
\lim \limits_{k\to +\infty}\lambda_{k,\pm}=0,
$
we know that there exists $k_0>0$ such that for $k\ge k_0$, there holds
$$
0<\cdots<\lambda_{k,+}<\lambda_{k,-}<\cdots<\lambda_{k_0,+}<\lambda_{k_0,-}\le\lambda_0.
$$
For all $k\ge k_0$, we have
\begin{eqnarray*}
\begin{aligned}
&\sin\left(-\frac{\sqrt{7}}{2}\log \lambda_{k,+}+c_8-c_2 \right)=(-1)^k\sin(\delta_0),\\
&\sin\left(-\frac{\sqrt{7}}{2}\log \lambda_{k,-}+c_8-c_2 \right)=(-1)^{k+1}\sin(\delta_0).
\end{aligned}
\end{eqnarray*}
By \eqref{long2}, we obtain
$$
\mathcal{F}[r_0](\lambda_{k,\pm})=\lambda_{k,\pm}^{\frac{1}{2}}\left\{
\pm(-1)^k \frac{c_1c_7\sqrt{7}}{2u_1(r_0)r_0^6}\sin(\delta_0)+O\left(\lambda_{k,\pm}^{\frac{1}{2}}\left(r_0^{-4}+\lambda_{k,\pm}^{\frac{3}{2}}r_0^{-\frac{1}{2}}\right)\right)\right\}.
$$
Since $\lim \limits_{k\to +\infty}\lambda_{k,\pm}=0,$ and $\delta_0>0$ is small enough, there exists $k_1\ge k_0$ such that, for any $k\ge k_1$, there holds
$$
\mathcal{F}[r_0](\lambda_{k,+})\mathcal{F}[r_0](\lambda_{k,-})<0.
$$
Due to that fact that the function $\lambda\to \mathcal{F}[r_0](\lambda)$ is continuous, then by the mean value theorem, for any $k\ge k_1$, there exists $\bar{\mu}_k$ such that
$$
\mathcal{F}[r_0](\bar{\mu}_k)=0,\ \ \ \bar{\mu}_k\in(\lambda_{k,+}, \lambda_{k,-}).
$$
Combining \eqref{exterios}, since $0<\bar{\mu}_k<\lambda_0$, we have
$
F[r_0](\varepsilon(\bar{\mu}_k),\mu_k)=0$ and $ \mathcal{F}[r_0](\bar{\mu}_k)=0,
$
i.e.,
$$
\Phi_{\text{ext}}[\varepsilon(\bar{\mu}_k)](r_0)=\Phi_{\text{int}}[\bar{\mu}_k](r_0),\ \ \Phi_{\text{ext}}[\varepsilon(\bar{\mu}_k)]'(r_0)=\Phi_{\text{int}}[\bar{\mu}_k]'(r_0).
$$
We define $\mu_{n}:=\bar{\mu}_{k+n}$.
For $k\ge k_1$ and $n\in\mathbb{N}$, the functions
\[\Phi_{n}(r):=\begin{cases}\Phi_{\text{int}}[\mu_{n}](r) &\text{for }0\leq r\leq r_0,\\\Phi_{\text{ext}}[\varepsilon(\mu_{n})](r) &\text{for }r>r_0.\end{cases}\]
are smooth radial solutions of \eqref{swin}.
\\
\textbf{Step 6.} The asymptotic behavior.
Recall from \eqref{lj} that
$$
\Phi_n=\Phi_*+\varepsilon(\mu_n)u_1(r)+\varepsilon(\mu_n) w(r),\ \ r\ge r_0,
$$
where
$
\lim \limits_{n\to +\infty}\varepsilon(\mu_n)=0.
$
By \eqref{sap}, \eqref{origin}, and \eqref{swi}, we have
$$
\sup_{r_0\le r\le1}r^{\frac{5}{2}}(r\partial_r+1)(|u_1|+|w|)+\sup_{ r\ge1}r^2(r\partial_r+1)(|u_1|+|w|)\lesssim1.
$$
Combining \eqref{swiland} and \eqref{origin}, we have
\begin{equation}\nonumber
\begin{aligned}
&\sup_{r\ge r_0}(1+r^2)|(r\partial_r+1)(\Phi_n-\Phi_*)|\\
&\lesssim
\varepsilon(\mu_n)\left(\sup_{r\ge r_0}(r\partial_r+1)(|u_1|+|w|)+\sup_{r\ge 1}r^2(r\partial_r+1)(|u_1|+|w|)\right)\lesssim \varepsilon(\mu_n)r_0^{-\frac{5}{2}},
\end{aligned}
\end{equation}
which implies
\begin{equation}\label{wat}
    \lim_{n\to+\infty}\sup_{r\ge r_0}(1+r^2)|(r\partial_r+1)(\Phi_n-\Phi_*)|=0.
\end{equation}
Thus, we complete the proof of \eqref{is1}
.

For the interior part estimate, for $0\le r\le r_0$, we know from \eqref{liuq} that
$$
\Phi_n=\frac{1}{\mu_n^2}(\bar{Q}+\mu_n^4 Q_1)\left(\frac{r}{\mu_n}\right),
$$
where
\begin{equation}\nonumber
\sup_{0\le r\le \frac{r_0}{\mu_n}}(1+r)^{-\frac{1}{2}}(|Q_1|+|r\partial_r Q_1|)\lesssim1.
\end{equation}
For $r\le r_0$, we have
$$
(r\partial_r+1)\left|\Phi_n-\frac{1}{\mu_n^2}\bar{Q}\left(\frac{r}{\mu_n}\right)\right|=\mu_n^2(r\partial_r+1)\left|Q_1\left(\frac{r}{\mu_n}\right)\right|\lesssim \mu_n^2\left(1+\frac{r}{\mu_n}\right)^{\frac{1}{2}}= \mu_n^{\frac{3}{2}}(\mu_n+r)^{\frac{1}{2}}.
$$
Then by
$
\lim \limits_{n\to +\infty}\mu_n=0,
$
we get
\begin{equation}\label{wat1}
\lim_{n\to+\infty}\sup_{r\le r_0}(r\partial_r+1)\left|\Phi_n-\frac{1}{\mu_n^2}\bar{Q}\left(\frac{r}{\mu_n}\right)\right|=0,
\end{equation}
which completes the proof of \eqref{is}.
\end{proof}
\section{Self-similar blow-up solutions}\label{Proofofmainresults}
We now give the proof of Theorem \ref{Main-theorem} for $d=3$. As mentioned previously, the proof for $d\in[4,9]$ is directly extendable.
\begin{proof}[Proof of Theorem \ref{Main-theorem}]
Recall from Proposition \ref{swin1} that  $\Phi_n$ are smooth radially symmetric solutions to equation \eqref{swin}.
By $\Phi_n=\frac{1}{2r^3}\int_0^rU_n(s)s^2ds,
$
we have
$
6\Phi_n+2r\partial_r\Phi_n=U_n.
$
It is clear that $U_n$ are radially symmetric solutions of \eqref{self-similar}.
By \eqref{wat},
we get
\begin{equation}\nonumber
\lim_{n\to+\infty}\sup_{r\ge r_0}(1+r^2)\left|U_n-\frac{2}{r^2}\right|=0.
\end{equation}
We know from \eqref{wat1} that
$$
\lim_{n\to+\infty}\sup_{r\le r_0}\left|U_n-\frac{1}{\mu_n^2}Q\left(\frac{r}{\mu_n}\right)\right|=0.
$$
This completes the proof of \eqref{wat2} and \eqref{wat3}.

For any $0<T<+\infty$, take $u_0=T^{-1}U_n\left(T^{-\frac{1}{2}}x\right)$. Since  $U_n(y)$ are self-similar profiles solve \eqref{self-similar}, the corresponding solution $u$ blows up in finite time $T$ with
\begin{equation}\label{hl0o}
u(x,t)=\frac{1}{T-t}U_n\left(\frac{x}{\sqrt{T-t}}\right).
\end{equation}
Because the functions $U_n$ are bounded, the blow-up is of type I.

We know from \eqref{lj} that \begin{equation}\label{jav}
U_n(y)\sim \frac{1}{|y|^2},\ \ \text{as}\  |y|\to+\infty.\end{equation}
Assume by contradiction that $B(u_0)\ne0$. But for any $\delta>0$ and $|x|\ge\delta$, we have
\begin{equation}\label{zqy}
\begin{aligned}
\lim_{t\to T}||u(x,t)||_{L^\infty(\mathbb{R}^3)}=
\lim_{t\to T}\left\|\frac{1}{T-t}U_n\left(\frac{x}{\sqrt{T-t}}\right)\right\|_{L^\infty(\mathbb{R}^3)}\lesssim \frac{1}{|x|^2}\le\frac{1}{\delta^2}<+\infty,
\end{aligned}
\end{equation}
which  contradicts the assumption $B(u_0)\ne0$. Therefore, the blow-up point of the solution $u(x,t)$ must be the origin, i.e., $B(u_0)=0$.

For any $\delta_1>0$, by \eqref{jav}, \eqref{zqy}, parabolic regularity and the Arzel\`{a}-Ascoli theorem, there exists a function $u^*$ such that
$$
\lim_{t\to T}u(x,t)\to u^*,\ \ \forall \ |x|\ge\delta_1,
$$
where $|u^*(x)|\sim\frac{1}{|x|^2}$.
For $p\in[1,\frac{3}{2})$, we get
$$
\lim_{t\to T}||u(t)-u^*||_{L^p(\mathbb{R}^3)}^p=\lim_{t\to T}\int_0^{\delta_1}|u(r,t)-u^*(r)|^pr^{2}dr\lesssim\int_0^{\delta_1}r^{2-2p}dr\to0,\ \text{as}\ \delta_1\to0,
$$
and \eqref{dzj} is proved.
This completes the proof of Theorem \ref{Main-theorem}.
\end{proof}


\begin{thebibliography}{10}

\bibitem{Zhou-2024}
X. Bai, M. Zhou, On the blow-up profile of Keller–Segel–Patlak system. Math. Ann., https://doi.org/10.1007/s00208-025-03102-z (2025).

\bibitem{Bedrossian-Masmoudi-2014}
J. Bedrossian, N. Masmoudi, Existence, uniqueness and lipschitz dependence for Patlak-Keller-Segel and Navier-Stokes in $\mathbb{R}^2$ with measure-valued initial data. Arch. Ration. Mech. Anal. (2014), 214(3): 717--801.

\bibitem{Biler-2018}
P. Biler, Singularities of solutions to chemotaxis systems, volume 6 of De Gruyter Series in Mathematics and Life Sciences. De Gruyter, Berlin, 2020.

\bibitem{Biler-2006}
P. Biler,  G. Karch,  P. Laurencot, and T. Nadzieja,  The 8$\pi$ problem for radially symmetric solutions
of a chemotaxis model in the plane. Math. Methods Appl. Sci.  (2006), 29(13): 1563--1583.

\bibitem{Blanchet-2006}
A. Blanchet,  J. Dolbeault, and B. Perthame, Two-dimensional Keller-Segel model: Optimal critical mass and qualitative properties of the solutions. Electron. J. Differential Equations (2006), 44: 1--32.


\bibitem{Blanchet-2008}
A. Blanchet, J.A. Carrillo, and N. Masmoudi, Infinite time aggregation for the critical Patlak-Keller-Segel model in $\mathbb{R}^2$. Commun. Pure Appl. Math. (2008), 61(10): 1449--1481.

\bibitem{Blanchet-Carrillo-Laurençot-2009}
A. Blanchet,  J.A. Carrillo, and P. Laurençot,  Critical mass for a Patlak–Keller–Segel model with degenerate diffusion in higher dimensions. Calc. Var. Partial Differential Equations (2009), 35(2): 133--168.

\bibitem{Michael}
M.P. Brenner, P. Constantin, L.P. Kadanoff, A. Schenkel, and S.C. Venkataramani,
Diffusion, attraction and collapse. Nonlinearity (1999), 12(4): 1071--1098.

\bibitem{Budd-Norbury}
C. Budd, J. Norbury, Semilinear elliptic equations and supercritical growth. J. Differential Equations (1987), 68(2): 169--197.

\bibitem{Buseghin-2023}
F. Buseghin, J. Davila, M. del Pino, and M. Musso, Existence of finite time blow-up in Keller-Segel system. arXiv:2312.01475, (2023).



\bibitem{Calvez-2012}
V. Calvez, L. Corrias, and M.A. Ebde, Blow-up, concentration phenomenon and global
existence for the Keller-Segel model in high dimension. Comm. Partial Differential Equations (2012), 37(4): 561--584.

\bibitem{Childress-Jerome}
S. Childress,  J.K. Percus, Nonlinear aspects of chemotaxis. Math. Biosci. (1981), 56(3-4): 217--237.

\bibitem{Childress}
S. Childress,  Chemotactic collapse in two dimensions, modelling of patterns in space and time (Heidelberg, 1983), 61--66. Lecture Notes in Biomath 55 (1984).


\bibitem{Collot-2022}
C. Collot, T. Ghoul, N. Masmoudi, and V.-T. Nguyen, Refined description and stability for singular solutions of the 2d Keller-Segel system. Comm. Pure Appl. Math. (2022), 75(7): 1419--1516.

\bibitem{Collot-2025}
C. Collot, T. Ghoul, N. Masmoudi, and V.-T. Nguyen, Singularity formed by the collision of two collapsing solitons in interaction for the 2{D} Keller-Segel system. arXiv:2409.05363 (2024).

\bibitem{Collot-2023}
C. Collot,  T. Ghoul, N. Masmoudi, and V.-T. Nguyen, Collapsing-ring blowup solutions for the Keller-Segel system in three dimensions and higher.  J. Funct. Anal. (2023), 285(7): 110065.

\bibitem{Collot-Pierre-2019}
C. Collot, P. Rapha$\mathrm{\ddot{e}}$l, and J. Szeftel, On the stability of type I blow up for the energy super critical heat equation. Mem. Amer. Math. Soc.  (2019), 260(1255): v+97.

\bibitem{Colllot-Zhang}
C. Collot, K. Zhang, On the stability of Type I self-similar blowups for the Keller-Segel system in three dimensions and higher. arXiv:2406.11358, (2024).


\bibitem{Corrias}
L. Corrias,  B. Perthame, and H. Zaag, Global solutions of some chemotaxis and angiogenesis
systems in high space dimensions. Milan J. Math. (2004), 72: 1--28.

\bibitem{Dancer-Guo-Wei}
E.N. Dancer, Z. Guo, and J. Wei, Non-radial singular solutions of the Lane-Emden equation
in $\mathbb{R}^N$. Indiana Univ. Math. J. (2012), 61(5): 1971--1996.

\bibitem{Davila-2020}
J. Davila, M. del Pino, J. Dolbeault, M. Musso, and J. Wei, Infinite time blow-up in the Patlak-Keller-Segel system: existence and stability. Arch. Ration. Mech.  Anal. (2024), 248(4): 61 .

\bibitem{Diaz-1998}
J.I. Diaz, T. Nagai and J.-M. Rakotoson, Symmetrization techniques on unbounded domains: application to a chemotaxis system on $\mathbb{R}^N$. J. Differential Equations (1998), 145(1): 156--183.

\bibitem{Dolbeault-2004}
J. Dolbeault, B. Perthame, Optimal critical mass in the two-dimensional Keller-Segel model in $\mathbb{R}^2$. C. R. Math. Acad. Sci. Paris (2004), 339(9): 611--616.

\bibitem{Fibich-2007}
G. Fibich, N. Gavish, and X.-P. Wang, Singular ring solutions of critical and supercritical nonlinear Schr$\mathrm{\ddot{o}}$inger equations. Phys. D (2007), 231(1): 55--86.

\bibitem{Ghoul-2018}
T. Ghoul, N. Masmoudi, Minimal mass blowup solutions for the Patlak-Keller-Segel equation. Commun. Pure Appl. Math. (2018), 71(10): 1957--2015.

\bibitem{Giga-Mizoguchi-Senba-2011}
Y. Giga,  N. Mizoguchi, and T. Senba, Asymptotic behavior of type I blowup solutions to a parabolic-elliptic system of drift-diffusion type. Arch. Ration. Mech. Anal.   (2011), 201(2): 549--573.

\bibitem{Glogic}
I. Glogi$\mathrm{\acute{c}}$, B. Sch$\mathrm{\ddot{o}}$rkhuber, Stable singularity formation for the Keller-Segel system in three dimensions. Arch. Ration. Mech. Anal. (2024), 248(1): 4.

\bibitem{Herrero-Medina-1997}
M.A. Herrero, E. Medina, and J.J.L. Velázquez, Finite-time aggregation into a single point in a reaction–diffusion system. Nonlinearity (1997), 10(6): 1739--1754.

\bibitem{Herrero-1998}
M.A. Herrero, E. Medina,  and J.J.L. Vel\'azquez,  Self-similar blow-up for a reaction-diffusion system. J. Comput. Appl. Math. (1998), 97(1-2): 99--119.

\bibitem{Herrero-Velázquez-1996}
M.A. Herrero, J.J.L. Vel\'azquez, Singularity patterns in a chemotaxis model. Math. Ann. (1996), 306(3): 583--623.

\bibitem{Herrero-Velázquez-1997}
M.A. Herrero, J.J.L. Vel\'zquez, A blow-up mechanism for a chemotaxis model. Ann. Scuola Norm. Sup. Pisa Cl. Sci.
(1997), 24(4): 633--683.

\bibitem{Horstmann-2003}
D. Horstmann, From 1970 until present: the Keller–Segel model in chemotaxis and
its consequences. I, Jahresber. Deutsch. Math.-Verein. (2003), 105(3): 103--165.

\bibitem{Horstmann-2004}
D. Horstmann, From 1970 until present: the Keller–Segel model in chemotaxis and
its consequences. II, Jahresber. Deutsch. Math.-Verein. (2004), 106(2): 51--69.

\bibitem{Hou-Liu-Wang-Wang}
Q. Hou, C. Liu, Y. Wang, and Z.-A. Wang, Stability of boundary layers for a viscous hyperbolic system arising from chemotaxis: one-dimensional case.
SIAM J. Math. Anal. (2018), 50(3): 3058--3091.

\bibitem{Keller-Segel}
E.F. Keller, L.A. Segel,  Initiation of slime mold aggregation viewed as an instability. J. Theor. Biol. (1970), 26(3): 399--415.

\bibitem{Kurokiba-2003}
M. Kurokiba,  T. Ogawa, Finite time blow-up of the solution for a nonlinear parabolic equation of drift-diffusion type.
Differential Integral Equations (2003), 16(4): 427--452.

\bibitem{Li-Zhou-2025}
Z. Li, T. Zhou, Nonradial stability of self-similar blowup to Keller-Segel equation in three dimensions. arXiv:2501.07073, (2025).

\bibitem{Mellet-2024}
A. Mellet,  Hele-Shaw flow as a singular limit of a Keller-Segel system with nonlinear diffusion. Calc. Var. Partial Differential Equations (2024), 63(8): 214.

\bibitem{Merle-2014}
F. Merle, P. Rapha$\mathrm{\ddot{e}}$l, and  J. Szeftel, On collapsing ring blow-up solutions to the mass supercritical nonlinear Schr$\mathrm{\ddot{o}}$inger equation. Duke Math. J. (2014), 163(2): 369--431.


\bibitem{Mizo-Senba}
N. Mizoguchi, T. Senba, Type II blowup solutions to a parabolic-elliptic system. Adv. Math. Sci. Appl. (2007), 17(0): 505--545.


\bibitem{Mizoguchi-2022}
N. Mizoguchi, Refined asymptotic behavior of blowup solutions to a simplified chemotaxis system. Commun. Pure Appl. Math. (2022), 75(8): 1870--1886.

\bibitem{Mizoguchi-2011}
N. Mizoguchi, T. Senba, A sufficient condition for type I blowup in a parabolic-elliptic system. J. Differential Equations (2011), 250(1): 182--203.


\bibitem{Nagai-1995}
T. Nagai, Blow-up of radially symmetric solutions to a chemotaxis system. Adv. Math. Sci. Appl. (1995), 5(2): 581--601.

\bibitem{Naito-2008}
Y. Naito, T. Suzuki, Self-similarity in chemotaxis systems. Colloq. Math. (2008), 111(1): 11--34.

\bibitem{Naito-Senba-2012}
Y. Naito, T. Senba, Blow-up behavior of solutions to a parabolic-elliptic system on higher dimensional domains. Discrete Contin. Dyn. Syst. (2012), 32(10): 3691--3713.

\bibitem{Nguyen-Zaag-2023}
V.-T. Nguyen, N. Nouaili, and H. Zaag, Construction of type I-Log blowup for the Keller-Segel system in dimensions $3$ and $4$.  arXiv:2309.13932, (2023).

\bibitem{Olver-Lozier}
F.W. Olver, D.W. Lozier, R.F. Boisvert, and C.W. Clark, NIST handbook of
mathematical functions. Cambridge university press, 2010.

\bibitem{Raphael-2014}
P. Raphael, R. Schweyer, On the stability of critical chemotactic aggregation. Math. Ann. (2014), 359(1-2): 267--377.

\bibitem{Souplet-Winkler}
P. Souplet, M. Winkler,  Blow-up profiles for the parabolic-elliptic Keller-Segel system in dimensions $n\ge3$.
 Commun. Math. Phys. (2019), 367(2): 665--681.

\bibitem{Suzuki-2011}
T. Suzuki, T. Senba,  Applied analysis, second edition,
Imp. Coll. Press, London, 2011 World Sci. Publ., Hackensack, NJ, 2011

\bibitem{Tao-Wang-2013}
Y. Tao, Z.-A. Wang, Competing effects of attraction vs. repulsion in chemotaxis. Math. Models Methods Appl. Sci.
(2013), 23(1): 1--36.

\bibitem{Velázquez-2002}
J.J.L. Velázquez, Stability of some mechanisms of chemotactic aggregation. SIAM J. Appl. Math. (2002), 62(5): 1581--1633,

\bibitem{Wolansky}
G. Wolansky,  On steady distributions of self-attracting clusters under friction and fluctuations. Arch. Ration. Mech. Anal. (1992), 119(4): 355--391.

\bibitem{Winkler-2019}
M. Winkler, How unstable is spatial homogeneity in Keller-Segel systems? A new critical mass phenomenon in two- and higher-dimensional parabolic-elliptic cases. Math. Ann. (2019), (373): 1237--1282.
\end{thebibliography}
\end{document}